\documentclass[11pt]{amsart}

\usepackage{amsmath}
\usepackage{amssymb}
\usepackage{graphicx}

\newcommand{\R}{{\mathbf R}}
\newcommand{\Z}{{\mathbf Z}}
\newcommand{\N}{{\mathbf N}}

\newcommand    {\e}{{\mathbf x}}
\newcommand    {\by}{{\mathbf y}}

\newcommand    {\bs}{{\mathbf s}}

\newcommand    {\bz}{{\mathbf z}}

\newcommand    {\m}{{\rm sign}}

\def\<{\langle}
\def\>{\rangle}

\newtheorem{theorem}{Theorem}[section]
\newtheorem{corollary}[theorem]{Corollary}
\newtheorem{lemma}[theorem]{Lemma}
\newtheorem{proposition}[theorem]{Proposition}
\theoremstyle{definition}
\newtheorem{definition}[theorem]{Definition}
\newtheorem{example}[theorem]{Example}
\newtheorem{remark}[theorem]{Remark}
\newtheorem{problem}[theorem]{Problem}

\numberwithin{equation}{section}

\title[Discreteness of the spectrum for SChr\"odinger operator]{Conditions of discreteness of the spectrum for
			 Schr\"odinger operator and some optimization problems for capacity and measures}

\author[L. Zelenko]{Leonid Zelenko}

\address[L. Zelenko]{Department of Mathematics \\
	University of Haifa  \\
	Haifa 31905  \\
	Israel}
\email{\tt zelenko@math.haifa.ac.il}



\begin{document}

	\begin{abstract}
		For the the Schr\"odinger operator $H=-\Delta+ V(\e)\cdot$, acting in the space $L_2(\R^d)\,(d\ge 3)$, with $V(\e)\ge 0$ and $V(\cdot)\in L_{1,loc}(\R^d)$, we obtain some constructive conditions
		for discreteness of its spectrum. Basing on the Mazya-Shubin criterion for discreteness of the spectrum of $H$  and using the isocapacity inequality and the concept of base polyhedron 
		for the harmonic capacity, we have estimated from below 
		the cost functional of an optimization problem, involved in this criterion, replacing a submodular constrain (in terms of the harmonic capacity) by a weaker but additive constrain (in terms of a measure). By this way we obtain  
		an optimization problem, which can be considered as an infinite-dimensional analogue of the optimal covering problem. We have solved this problem for the case of a non-atomic measure. This approach 
		enables us to obtain for the operator $H$  some sufficient conditions for discreteness of its spectrum in terms of non-increasing rearrangements, with respect to measures from the base polyhedron, 
		for some functions connected with the  potential $V(\e)$. We construct some counterexamples, which permit to compare our results between themselves and with results of other authors.
	\end{abstract}

	\maketitle
	
	
	\section{Introduction} \label{sec:introduction}
	\setcounter{equation}{0}
	
	In the present paper we consider the Schr\"odinger operator $H=-\Delta+ V(\e)\cdot$,
	acting in the space $L_2(\R^d)\,(d\ge 2)$. In what follows we
	assume that $d\ge 3$, $V(\e)\ge 0$ and $V(\cdot)\in L_{1,loc}(\R^d)$. Physically $V(\e)$ is the 
	potential of an external electric field.

	A wide literature is dedicated to investigation 
	of the spectrum for this operator. In particular, the case of discrete spectrum is interesting where it consists only of isolated eigenvalues of $H$
	with finite multiplicities. With the point of view of Quantum Mechanics in this case an electron can move only in a compact neighborhood of an atom kernel along an orbit 
	from a discrete collection of orbits. The simplest condition for  discreteness
	of the spectrum is: $\lim_{|\e|\rightarrow\infty}V(\e)=\infty$ (\cite{6}).  Physically this means that there is a potential barrier at infinity. 
	The first sufficient and necessary condition for discreteness of the spectrum to $H$ in the case of a semi-bounded below potential was obtained by A.M. Molchanov \cite{14}. 
	For the one-dimensional case ($d=1$) this criterion has the form: for any $r>0$	$\lim_{|x|\rightarrow\infty}\int_x^{x+r}V(t)\,\mathrm{d}t=\infty$. But in the multi-dimensional case ($d\ge 2$)
	the Molchaniv's criterion is more complicated, because it involves so called ``negligible'' subsets of $\R^d$, i.e., ones having a small harmonic capacity.
	In the paper \cite{13} V. Mazya and M. Shubin have improved significantly the Molchanov's result. In order to formulate the result from \cite{13}, let us introduce some notations. 
	Consider in $\R^d$ an open domain $\mathcal{G}$ satisfying the conditions:
	
	(a) $\mathcal{G}$ is bounded and star-shaped with respect to any point of an open ball $B_\rho(0)\,(\rho>0)$ contained in $\mathcal{G}$;
	
	(b) $\mathrm{diam}(\mathcal{G})=2$.
	
	As it was noticed in \cite{13}, condition (a) implies that $\mathcal{G}$ can be represented in the form
	\begin{equation}\label{formofcalG}
	\mathcal{G}=\{\e\in\R^d:\,\e=r\omega,\, |\omega|=1,\,0\le r<r(\omega)\}, 
	\end{equation} 
	where $r(\omega)$ is a positive Lipschitz function on the standard unit sphere $S^{d-1}\subset\R^d$. For $r>0$
	and $\by\in\R^d$ denote 
	\begin{equation}\label{dfcalGry}
	\mathcal{G}_r(\by):=\{\e\in\R^d:\,r^{-1}\e\in\mathcal{G}\}+\{\by \}.
	\end{equation}
	Denote by
	$\mathcal{N}_{\gamma}(\by,r)\;(\gamma\in(0,1))$ the set of all
	compact sets $F\subseteq\bar{\mathcal{G}}_r(\by)$ satisfying the condition
	\begin{equation}\label{defNcap}
	\mathrm{cap}(F)\le\gamma\,\mathrm{cap}(\bar{\mathcal G}_r(\by)),
	\end{equation}
	where $\mathrm{cap}(F)$ is the harmonic capacity.
	\begin{theorem}\cite{13}\label{thMazSh}
	The spectrum of the operator $H$ is discrete, if  
	for some $r_0>0$ and for any $r\in(0,r_0)$ the condition
	\begin{equation}\label{cndmolch1}
	\lim_{\by\rightarrow\infty}\inf_{F\in\mathcal{N}_{\gamma(r)}(\by,r)}\int_{{\mathcal G}_r(\by)\setminus
		F}V(\e)\, \mathrm{d}\e=\infty,
	\end{equation}
	is satisfied, where 
	\begin{equation}\label{cndgammar}
	\forall\,r\in(0,\,r_0):\;\gamma(r)\in(0,1)\quad\mathrm{and}\quad
	\limsup_{r\downarrow 0} r^{-2}\gamma(r)=\infty,
	\end{equation}	
	\end{theorem}	
In \cite{13} also a necessary condition for discreteness of the spectrum was obtained, which is close to sufficient one. 

Notice that condition \eqref{cndmolch1} of Theorem \ref{thMazSh} is hardly verifiable, because in order to test it, one needs to solve a difficult optimization problem, whose cost functional is the set function  $\mathcal{I}(F)=\int_{{\mathcal G}_r(\by)\setminus F}V(\e)\, \mathrm{d}\e$ and the constrain 
$F\in\mathcal{N}_{\gamma(r)}(\by,r)$ is submodular (because``cap'' is a submodular set function (definition \eqref{submod}). In the present paper we estimate the cost functional from below using the isocapacity inequality \eqref{isocapineq} and replacing $F\in\mathcal{N}_{\gamma(r)}(\by,r)$ by a weaker but additive constrain. To this end we also use the concept of base polyhedron for the harmonic capacity (definition \eqref{defcore}). By this way on the base of  Theorem \ref{thMazSh} we obtain some sufficient conditions for discreteness of the spectrum in terms of measures, which permit a reformulation in terms of non-increasing rearrangements of some functions connected with the potential $V(\e)$.

Let us notice that in the papers \cite{1}, \cite{20}, \cite{9} and \cite{4} some constructive sufficient conditions for discreteness of the spectrum for $H$ have been found. In \cite{9} also the case of a unbounded below potential was studied and in \cite{4} also the case of a matrix-valued potential was considered. For a scalar non-negative potential  $V(\e)$
 the result of \cite{4} is most general among all these conditions, but our results are more general (see Remarks \ref{remGMD}, \ref{remonexampl} and Example \ref{ex4}). 

Let us make an overview of our results. Theorem \ref{thmeasinstcapcalG} is obtained by a direct use of the isocapacity inequality and yields a sufficient condition of discreteness of the spectrum for $H$ in terms of an optimization problem involving Lebesgue measure instead of the capacity.  This problem can be considered as an infinite-dimensional analogue of the optimal covering problem \cite{15}. We formulate it in a more general form - for any non-negative measure (Problem \ref{extrprobl}), with the purpose to use its solution for derivation of results involving the base polyhedron of harmonic capacity. Theorem \ref{prsolextrprob} solves Problem \ref{extrprobl} for a {\it non-atomic} measure (\cite{8}). In Proposition \ref{lmJR} we obtain a two-sided estimate for the solution of Problem \ref{extrprobl}, obtained 
in Theorem \ref{prsolextrprob}, via a non-increasing rearrangement of the function $W(\e)$ taking part in the cost functional. This estimate enables us to formulate most of our conditions of discreteness of the spectrum for $H$ in terms of non-increasing rearrangements. In such formulations these conditions are easier verifiable than ones obtained by a direct use of Theorem \ref{prsolextrprob}. Theorem \ref{lmmeasinstcapnew} yields a condition for discreteness of the spectrum based immediately on Theorems \ref{thmeasinstcapcalG}, \ref{prsolextrprob} and Theorem 
\ref{lmmeasinstcap1} is reformulation of it in terms of the rearrangement of the potential $V(\e)$ 
with respect to Lebesgue measure.

All the rest our results are based on the concept of base polyhedron for harmonic capacity. 
Condition for discreteness of the spectrum given in Theorem \ref{thusecore1} involves the optimization problem, formulated in Problem \ref{extrprobl}, with measures from the base polyhedron.
Condition for discreteness of the spectrum given in Theorem \ref{thusecore2} involves only  measures from the base polyhedron, which are equivalent to Lebesgue measure, and it is formulated in terms of non-increasing rearrangements, with respect to these measures, for products of $V(\e)$ and the densities of Lebesgue measure with respect to them. Corollary \ref{thusecoredstort} is based on Theorem \ref{thusecore2} and Proposition \ref{lmdescribedens}, where a constructive description of a part ${\mathrm M}_f(\by,r)$ of the base polyhedron is obtained
(see also Definition \ref{dfMacry}). In turn, Proposition \ref{lmdescribedens} rests on  isocapacity inequality \eqref{isocapineq} and the description, obtained in \cite{2} for the core of convex distortions of probability measures. Theorem \ref{thlogmdenspolyhed}, based on Corollary \ref{thusecoredstort}, yields a easier verifiable  condition for discreteness of the spectrum for $H$
in terms of rearrangements of the potential $V(\e)$ with respect to Lebesgue measure on cubes from $m$-adic partitions of unit cubes. In the formulation of Theorem \ref{thlogmdenspolyhed} we use our concept of  $(\log_m,\,\theta)$- dense system of subsets of a unit cube (Definition \ref{dfdenslogmtetpart1}), which permits to place cubes from $m$-adic partitions by their sizes. We need this placement because, roughly speaking, on these cubes  
the lower bound of relative Lebesgue measure of the sets of points, where $V(\e)$ tends to infinity as the cubes go to infinity, depends on their sizes. The property of $(\log_m,\,\theta)$-density ensures that on balls with any centers and arbitrary small radius a similar dependence on the radius is conserved for the lower bound of relative measure of the sets, mentioned above.  This circumstance enables us to use Corollary \ref{thusecoredstort} in the proof of Theorem \ref{thlogmdenspolyhed}.

The paper is organized as follows. After this Introduction, in Section \ref{sec:prel} (Preliminaries)  we introduce some concepts and notations used in the paper. In Section
\ref{sec:meinres} we formulate the main results, in Section \ref{sec:proofmainres} 
we prove them. In Section \ref{sec:examples} we construct examples of $(\log_m,\,\theta)$- dense systems of sets and some counterexamples, which permit to compare our results between themselves and with results of other authors. Section \ref{sec:appendix} is Appendix, where we formulate and prove Proposition \ref{prcore} on existence of the non-empty base polyhedron for harmonic capacity 
and Proposition \ref{lmdescribedens} mentioned above.

 \section{Preliminaries} \label{sec:prel}
\setcounter{equation}{0}

Let us come to agreement on some notations and terminology. Let $\Omega$ be an open and bounded domain in $R^d$. We denote by
$\Sigma_B(\bar\Omega)$ the $\sigma$-algebra of all
Borel subsets of $\bar\Omega$. By $\Sigma_L(\bar\Omega)$ we denote the $\sigma$-algebra of all Lebesgue measurable subsets of $\bar\Omega$, i.e., it is the Lebesgue completion of $\Sigma_B(\bar\Omega)$ by the Lebesgue measure $\mathrm{mes}_d$. If $(X,\Sigma,\mu)$ is a measure space, we call all sets from $\Sigma$ $\,\mu$-{\it measurable} and if $X=\bar\Omega\subseteq\R^d$, $\mu=\mathrm{mes}_d$ and $\Sigma=\Sigma_L(\bar\Omega)$, we simply call them  measurable. If a measure is absolutely continuous with respect to Lebesgue measure, we simply call it absolutely continuous. By $B_r(\by)$ we denote the open ball in $\R^d$ whose radius and center are $r>0$ and $\by$.

Let us recall the definition of the {\it harmonic (or Newtonian) capacity}\footnote[1]{In the Russian literature it is often called  \it{Wiener capacity.}}  of a compact set
$E\subset\R^d$ (\cite{13}):
\begin{eqnarray}\label{dfWincap}
&&\mathrm{cap}(E):=\inf\Big(\big\{\int_{\R^d}|\nabla u(\e)|^2\,
\mathrm{d}\e\,:\;u\in C^\infty(\R^d), \;u\ge
1\;\mathrm{on}\;E,\\
&&u(\e)\rightarrow 0\;\mathrm{as}\;|\e|\rightarrow\infty \big\}\Big).\nonumber
\end{eqnarray}
It is known \cite{3} that the set function ``cap'' can be extended in a suitable manner from the set of all compact subsets of the space $\R^d$ to the set of all Borel subsets of it.
 It is known  (\cite{12}, \cite{11}) that the set function ``cap''
is {\it submodular} (concave) in the sense that for any pair of sets $A,\,B\in\Sigma_B(\bar\Omega)$
\begin{equation}\label{submod}
\mathrm{cap}(A\cup B)+\mathrm{cap}(A\cap
B)\le\mathrm{cap}(A)+\mathrm{cap}(B).
\end{equation}
and the {\it isocapacity inequality} is valid:
\begin{equation}\label{isocapineq}
\forall\,F\in\Sigma_B(\bar\Omega):\quad\mathrm{mes}_d(F)\le
c_d\,(\mathrm{cap}(F))^{d/(d-2)}
\end{equation}
with $c_d=\big(d(d-2)(\mathrm{mes}_d(B_1(0)))^{2/d}\big)^{-d/(d-2)}$, which  comes as
identity if $F$ is a closed ball.

By $M(\bar\Omega)$ denote the set of all  additive set functions on
$\Sigma_B(\bar\Omega))$ (we shall call them briefly ``measures'') and by $M^+(\bar\Omega)$ denote the set of
all non-negative measures from $M(\bar\Omega))$. In the theory of
coalition games (\cite{18} \cite{16}, \cite{10}) the concept of the
{\it core} of a game is used. 
Following to
\cite{5}, we define for the harmonic capacity on $\bar\Omega$ a dual concept of the {\it base
	polyhedron} $\mathrm{BP}(\bar\Omega)$: 
\begin{eqnarray}\label{defcore}
&&\mathrm{BP}(\bar\Omega):=\{\mu\in M^+(\bar\Omega)):\;\\
&&\mu(A)\le \mathrm{cap}(A)\;\mathrm{for\; all}\;
A\in\Sigma_B(\bar\Omega))\;\mathrm{and}\;\mu(\bar\Omega)=\mathrm{cap}(\bar\Omega)\}.\nonumber
\end{eqnarray}
This set is nonempty, convex, and compact in the weak*-topology
(Proposition \ref{prcore}). If $\Omega=\mathcal{G}_r(\by)$ (see \eqref{dfcalGry}. \eqref{formofcalG}), we shall
write briefly $M^+(\by,r)$ and $\mathrm{BP}(\by,r)$. Denote by
$\mathrm{BP}_{eq}(\bar\Omega)$ the subset of 
$\mathrm{BP}(\bar\Omega)$ 
consisting of Radon measures which are equivalent to
Lebesgue measure $\mathrm{mes}_d$ (\cite{7}).

Let us recall the concept
of {\it measure preserving mapping} (\cite{2}, Definition
2). 
\begin{definition}\label{dfmeaspres}
Let $(\Omega,\,\mathcal{A},\,Q)$ be a  probability space
and $\lambda$ be Lebesgue measure on $[0,1]$. A measurable
function $s:\,\Omega\rightarrow [0,1]$ is called {\it measure
	preserving} , if $\lambda(B)=Q\big(s^{-1}(B)\big)$ for any Borel
subset of $[0,1]$. Denote by $\mathcal{S}(\Omega,\,Q)$ the
collection of all such functions.	
\end{definition}
 
Suppose that $\Omega=\bar
B_r(\by)$, $\mathcal{A}=\Sigma_L(\bar B_r(\by))$ 
and $Q$ is the normalized
Lebesgue measure $m_{d,r}$ on $\bar B_r(\by)$, defined by:
\begin{equation}\label{dfmdr}
m_{d,r}(A):=\frac{\mathrm{mes}_d(A)}{\mathrm{mes}_d(B_r(\by))}\quad(A\in\Sigma_L(\bar B_r(\by))).
\end{equation}
 We shall need the following set of measures, which is a part of $\mathrm{BP}_{eq}(\bar
B_r(\by))$ (Proposition \ref{lmdescribedens}):
\begin{definition}\label{dfMacry}
	Consider the function $f(t)=t^{(d-2)/d}\;(t\in[0,1])$ and denote by ${\mathrm M}_f(\by,r)$ the set of absolute
	continuous measures on $\bar B_r(\by)$, whose densities run over
	the following convex set:
	\begin{equation}\label{dfcalCo}
	{\mathcal Co}\,(\by,\,r):=\mathrm{cap}(\bar
	B_r(0))\cdot\,\overline{\mathrm{co}}\Big(\{f^\prime\,\circ\,s:\,s\in\mathcal{S}(\bar
	B_r(\by),\,m_{d,r})\}\Big),
	\end{equation}
	where  $``\mathrm{co}''$ denotes the convex
	hull and the closure is taken for the $L_1(\bar B_r(\by),\\m_{d,r})$
	topology.
\end{definition}
\begin{remark}\label{remoncubes}
	 We shall consider the unit cube $Q=[-1,\,1]^d\subset\R^d$ and the dilation of it $Q_r:=r\cdot Q\,(r>0)$ and the translation of the latter 
	$Q_r(\by):=Q_r+\{\by \}$.	
\end{remark}
Denote by $[x]$ the integer part of a real number $x$. In the sequel we need the following concepts:
\begin{definition}\label{dfregpar1}
	We call a subset of $\R^d$ a {\it regular parallelepiped}, if it has the form $\times_{k=1}^d[a_k,\,b_k]$.
\end{definition}
\begin{definition}\label{dfdenslogmtetpart1}
	Suppose that $m>1$	and $\theta\in(0,1)$. A sequence $\{D_n\}_{n=1}^\infty$ of subsets of a cube $Q_1(\by)$ is said to be a {\it $(\log_m,\,\theta)$- dense system} in $Q_1(\by)$, if 
	
	(a) each $D_n$ is a finite union of regular parallelepipeds;
	
	(b) for any 
	cube $Q_r(\bz)\subseteq Q_1(\by)$ with $r\in\big(0,\,\min\{1,\,\frac{1}{\theta m^2}\}\big)$ there is  
	$j\in\{1,2,\dots, \big[\log_m\big(\frac{1}{\theta r}\big)\big]\}$
	such that for some regular parallelepiped $\Pi\subseteq D_j$ 
	there is a cube $Q_{\theta r}(\bs)$, contained in $\Pi\cap Q_r(\bz)$.
\end{definition}

\section{Main results} \label{sec:meinres}
\setcounter{equation}{0}

Denote by $\mathcal{M}_{\gamma}(\by,r)\;(\gamma\in(0,1))$ the
collection of all Borel sets $F\subseteq\mathcal{G}_r(\by)$
satisfying the condition
$\mathrm{mes}_d(F)\le\gamma\,\mathrm{mes}_d(\mathcal{G}_r(\by))$, where the domain $\mathcal{G}_r(\by)$ is defined by 
\eqref{dfcalGry}, \eqref{formofcalG}. A direct use of isocapacity inequality \eqref{isocapineq}
 leads to the following claim:
\begin{theorem}\label{thmeasinstcapcalG}
	Suppose that for some $r_0>0$ and any $r\in(0,r_0)$ the condition
	\begin{equation}\label{cndmeasinstcapcalG}
	\lim_{|\by|\rightarrow\infty}\inf_{F\in\mathcal{M}_{\tilde\gamma(r)}(\by,r)}\int_{\mathcal{G}_r(\by)\setminus
		F}V(\e)\, \mathrm{d}\e=\infty
	\end{equation}
	is satisfied with a function $\tilde\gamma(r)$ satisfying the condition 
	\begin{equation}\label{cndtildgam}
	\forall\,r\in(0,\,r_0):\;\tilde\gamma(r)\in(0,1)\quad\mathrm{and}\quad\limsup_{r\downarrow
		0}\,r^{-2(d-2)/d}\,\tilde\gamma(r)=\infty.
	\end{equation}
	Then the spectrum of the operator $H=-\Delta+V(\e)$ is discrete.
\end{theorem}

In order to represent the expression in left hand side of \eqref{cndmeasinstcapcalG} in a more constructive form,  
we need to solve the following optimization problem:
\begin{problem}\label{extrprobl}
	Let $(X,\,\Sigma,\,\mu)$ be a measure space with a non-negative measure
	 $\mu$  and $W(x)$ be a
	non-negative function defined on $X$ and belonging to
	$L_1(X,\,\mu)$. For $t\in(0,\,\mu(X))$ consider the
	collection $\mathcal{E}(t,X,\,\mu)$ of all $\mu$-measurable
	sets $E\subseteq X$ such that $\mu(E)\ge t$. The goal is to find
	the quantity
	\begin{equation}\label{dfIVOm}
	I_W(t,\,X,\,\mu)=\inf_{E\in\mathcal{E}(t,X,\,\mu)}\int_E W(x)\,
	\mu(\mathrm{d}x).
	\end{equation}
\end{problem}

In the formulation of next claim we shall use the following notations.
For the measure space and the function $W(x)$, introduced in Problem \ref{extrprobl},
consider the quantity:
\begin{eqnarray}\label{dfJVOm}
&&J_W(t,\,X,\mu):=\int_{\mathcal{K}_W^-\big(t,\,X,\,\mu\big)}W(x)\,
\mu(\mathrm{d}x)+\\
&&\big(t-\kappa_W^-(t,\,X,\,\mu)\big) W_\star(t,\,X,\,\mu),\nonumber\end{eqnarray}
where $W_\star(t,\,X,\,\mu)$ is the non-decreasing rearrangement of the function\\ $W(x)$, i.e.,
\begin{equation}\label{dfstOm}
W_\star(t,\,X,\,\mu):=\sup\{s>0:\;\lambda_\star(s,\,W,\,X,\,\mu)<
t\}\quad (t>0)
\end{equation}
with
\begin{equation}\label{dfFVOm}
\lambda_\star(s,\,W,\,X,\,\mu)=\mu({\mathcal L} _\star(s,W,X)),
\end{equation}
\begin{equation}\label{dfKVOm}
{\mathcal L} _\star(s,W,X)=\{x\in X:\;W(x)\le s\}.
\end{equation}
Furthermore,
\begin{equation}\label{dfKWt}
\mathcal{K}_W^-\big(t,\,X,\,\mu\big)={\mathcal L} _\star(s^-,W,X)\vert_{s=W_\star(t,\,X,\,\mu)},
\end{equation}
and
\begin{equation}\label{dfFVommin}
\kappa_W^-(t,\,X,\,\mu)=\mu\big(\mathcal{K}_W^-(t,\,X,\,\mu)\big),
\end{equation}
where
\begin{equation}\label{dfKVsminOm}
{\mathcal L} _\star(s^-,W,X)=\bigcup_{u<s}{\mathcal L} _\star(u,W,X)=\{x\in
X:\;W(x)< s\},
\end{equation}

The following claim solves Problem \ref{extrprobl} for a non-atomic  measure:
\begin{theorem}\label{prsolextrprob}
	 Suppose that, in addition to  conditions of Problem \ref{extrprobl}, the measure 
	$\mu$ is non-atomic. Then
	
	(i) for any $t\in(0,\,\mu(X))$ there exists a $\mu$-measurable set $\tilde{\mathcal K}\subseteq X$ such that 
	\begin{equation}\label{estmesKV}
	\mu(\tilde{\mathcal K})=t,
	\end{equation}
	for the quantity $J_W(t,\,X,\,\mu)$, defined by
	\eqref{dfJVOm}-\eqref{dfFVommin}, the representation
	\begin{equation}\label{reprJV}
	J_W\big(t,\,X,\,\mu\big)= \int_{\tilde{\mathcal
			K}}W(x)\, \mu(\mathrm{d}x)
	\end{equation}
	is valid and
	\begin{eqnarray}\label{propKV}
	&&\forall\;x\in\tilde{\mathcal K}:\quad W(x)\le
	W_\star(t,\,X,\,\mu),\\
	&&\forall\;x\in X\setminus\tilde{\mathcal K}:\quad W(x)\ge W_\star(t,\,X,\,\mu);\nonumber
	\end{eqnarray}
	
	(ii) the equality
	\begin{equation}\label{IeqJ}
	I_W(t,\,X,\,\mu)=J_W(t,\,X,\,\mu)
	\end{equation}
	is valid.
	\end{theorem}

	In the next claim we shall obtain a two-sided estimate for the solution $J_W(t,\,X,\,\mu)$ of Problem \ref{extrprobl}  via a non-increasing rearrangement of the function $W(x)$ on  $X$. 
	This rearrangement is following:
	\begin{equation}\label{dfSVyrdel}
	\bar W^\star(t,X,\mu):=\sup\{s>0\,:\;\lambda^\star(s,\,W,\,X,\,\mu)\ge
	t\}\quad (t>0)\footnote[1]{In this notation we use the accent ``bar'', because in the literature ones denote by $W^\star$ the non-increasing rearrangement with the strong inequalities in its 
		definition}, 
	\end{equation}
	where
	\begin{eqnarray}\label{dfLVsry}
	&&\lambda^\star(s,\,W,\,X,\,\mu)=\mu(\mathcal{L}^\star(s,\,W,\,X,\,\mu)),\\
	&&\mathcal{L}^\star(s,\,W,\,X)=\{x\in X,:\; W(x)\ge s\}.\nonumber
	\end{eqnarray}

The promised claim is following:
	\begin{proposition}\label{lmJR}
	 	Suppose that, in addition to conditions of Problem \ref{extrprobl} and Theorem
		\ref{prsolextrprob}, the measure $\mu$ is finite.
		Then for $\theta>1$	and $t\in(0,\,\mu(X))$ the estimates
		\begin{equation}\label{estJR1}
		J_W(\mu(X)-t/\theta,X,\mu))\ge\frac{(\theta-1)t}{\theta}\bar W^\star(t,X,\mu),
		\end{equation}
		\begin{equation}\label{estJR2}
		J_W(\mu(X)-t,\,X,\mu)\le (\mu(X)-t) \bar W^\star(t,X,\mu)
		\end{equation}
		are valid.
	\end{proposition}

	\begin{remark}\label{rembriefnot}
	If $X=\bar\Omega$, where $\Omega$ is an open domain in $\R^d$ and $\mu=\mathrm{mes}_d$, we shall omit $\mu$ in the  notations introduced above, i.e., to write $J_W(t,\bar\Omega)$, $W_\star(t,\bar\Omega)$, $\bar W^\star(t,\bar\Omega)$, $\lambda_\star(s,\,W,\,\bar\Omega)$ and  $\lambda^\star(s,\,W,\,\bar\Omega)$. 
 	In the case where $\Omega=\mathcal{G}_r(\by)$  we shall write $J_W(t,\by,r,\mu)$, $W_\star(t,\by,r,\mu)$, $\bar W^\star(t,\by,r,\mu)$, $\lambda_\star(s,\,W,\,\by,r,\,\mu)$ and $\lambda^\star(s,\,W,\,\by,r,\,\mu)$. 
 	If $\mathcal{G}_r(\by)$ is a ball $B_r(\by)$, we shall use the same notations,F if they could not lead to a confusion.
	If $\mu=\mathrm{mes}_d$, we shall omit $\mu$ in the last notations too.	
\end{remark}

On the base of Theorems \ref{thmeasinstcapcalG} and \ref{prsolextrprob} we obtain the following claim:
	
\begin{theorem}\label{lmmeasinstcapnew}
	Suppose that 
	for some $r_0>0$ and any $r\in(0,r_0)$ the condition
	\begin{equation}\label{cndjVy}
	\lim_{|\by|\rightarrow\infty}J_V(\sigma(r),
	\by,r)=\infty
	\end{equation}
	is satisfied, where
	$\sigma(r)=(1-\tilde\gamma(r))\mathrm{mes}_d(\mathcal{G}_r(0))$ and
	$\tilde\gamma(r)$ satisfies conditions \eqref{cndtildgam}.
	Then the spectrum of the operator $H=-\Delta+V(\e)\cdot$ is
	discrete.
\end{theorem}	

Proposition \ref{lmJR} enables us to reformulate Theorem \ref{lmmeasinstcapnew} in terms of 
the non-increasing rearrangement of the potential $V(\e)$, defined above. 

\begin{theorem}\label{lmmeasinstcap1}
	Suppose that for some $r_0>0$ and any $r\in(0,r_0)$ the
	condition 
	\begin{equation}\label{cndSVy}
	\lim_{|\by|\rightarrow\infty}\bar V^\star(\hat\delta(r),\by,r)=\infty
	\end{equation}
	is satisfied with
	$\hat\delta(r)=\hat\gamma(r)\mathrm{mes}_d(\mathcal{G}_r(0))$ and
	$\hat\gamma(r)$ satisfies conditions \eqref{cndtildgam}.
	Then the spectrum of the operator $H=-\Delta+V(\e)$ is discrete.
\end{theorem}

\begin{remark}\label{remequivcond}
	Notice that condition \eqref{cndSVy} is easier verifiable than condition \eqref{cndjVy}. On the other hand,
	estimates \eqref{estJR1} and \eqref{estJR2} imply that these conditions are equivalent in the following sense: for some function $\tilde\gamma(r)$	satisfying conditions \eqref{cndtildgam} the condition \eqref{cndjVy} is satisfied with $\sigma(r)=(1-\tilde\gamma(r))\mathrm{mes}_d(\mathcal{G}_r(0))$ if and only if
	for some function $\hat\gamma(r)$	satisfying conditions \eqref{cndtildgam} the condition
	\eqref{cndSVy} is satisfied with 	$\hat\delta(r)=\hat\gamma(r)\mathrm{mes}_d(\mathcal{G}_r(0))$
\end{remark}

\begin{remark}\label{remGMD}
	From Theorem 6 of \cite{4} the following criterion of
	discreteness of the spectrum of the operator $H=-\Delta+V(\e)$
	 follows (in our notations): if for some numbers $\delta>0$,
	$c\in(0,1)$ and $r_0>0$ and for any $\by\in\R^d$, $r\in(0, r_0)$
	the condition
	\begin{equation}\label{cndGMD}
	\lambda^\star\Big(\frac{\delta}{\mathrm{mes}_d(Q_r(0))}\int_{Q_r(\by)}V(\e)\,d\e,V,\,r,\,\by\Big)\ge
	c\;\mathrm{mes}_d(B_r(0))
	\end{equation}
	is fulfilled, then the discreteness of the spectrum of the operator $H$ is
	equivalent to the condition:
	$\lim_{|\by|\rightarrow\infty}\int_{Q_r(\by)}V(\e)\,d\e=\infty$
	for some (hence for every) $r>0$. 
	Recall that the cube $Q_r(\by)$ is defined in Remark \ref{remoncubes}. 
	It is easy to see that as a sufficient condition  this
	criterion follows from Theorem \ref{lmmeasinstcap1} with $\mathcal{G}_r(\by)=Q_{r_1}(\by-r\vec a)$ 
	($r_1=2r/\sqrt{d}$, $\vec a=(d^{-1/2}, d^{-1/2},\dots,d^{-1/2})$), if one takes
	$\tilde\gamma(r)\equiv c$.  
	\end{remark}
 The concept of base polyhedron for harmonic capacity enables us to obtain some conditions 
 of discreteness of the spectrum for  Schr\"odinger operator, covering potentials which do not satisfy conditions of Theorem \ref{lmmeasinstcap1}. 
For $\mu\in
M^+(\by,r)$ denote by
$\mathcal{M}_\gamma^\mu(\by,r)\;(\gamma\in(0,1))$ the collection
of all Borel sets $F\subseteq\bar {\mathcal G}_r(\by)$ satisfying the
condition
\begin{equation}\label{defMmu}
\mu(F)\le\gamma\,\mu(\mathcal{G}_r(\by)).
\end{equation}
The following claim is valid:
\begin{theorem}\label{thusecore1}
	Suppose that for some $r_0>0$ and any $r\in(0,r_0)$
	\begin{eqnarray}\label{cndusecore4}
	\lim_{|\by|\rightarrow\infty}\;\sup_{\mu\,\in\,
		\mathrm{BP}(\by,r)}\;\inf_{F\in\mathcal{M}^\mu_{\gamma(r)}(\by,r)}\int_{\mathcal{G}r(\by)\setminus
		F}V(\e)\,\mathrm{d}\e=\infty,
	\end{eqnarray}
	where $\gamma(r)$ satisfies the condition \eqref{cndgammar}.
Then the spectrum of the operator $H=-\Delta+V(\e)$ is discrete.
\end{theorem}

Since each measure $\mu\in\mathrm{BP}_{eq}(\by,r)$ is equivalent to the Lebesgue measure
$\mathrm{mes}_d$, the Lebesgue completion of $\Sigma_B(\mathcal{G}r(\by))$ by $\mu$ coincides 
with $\Sigma_L(\mathcal{G}r(\by))$. Hence we can consider the complete measure space 
$\big(\mathcal{G}r(\by),\,\Sigma_L(\mathcal{G}r(\by)),\,\mu\big)$. 
For $\mu\in\mathrm{BP}_{eq}(\by,r)$
denote by $\alpha_\mu(\e)$ density of the measure
$\mathrm{mes}_d$ with respect to $\mu$, i.e.,
\begin{equation}\label{dfalphamu}
\alpha_\mu:=\frac{\mathrm{d}\,\mathrm{mes}_d}{\mathrm{d}\,\mu}.
\end{equation}

The following claim is based on Theorems \ref{thusecore1},
\ref{prsolextrprob} and Proposition \ref{lmJR}:
\begin{theorem}\label{thusecore2}
	Suppose that the condition 
	\begin{equation}\label{cndZmu1}
	\lim_{|\by|\rightarrow\infty}\;\sup_{\mu\in\mathrm{BP}_{eq}(\by,\,r)}\bar Z_{\mu}^\star(\psi_\mu(r),\,
	\by,\,r,\,\mu)=\infty
	\end{equation}
	is satisfied for some $r_0>0$ and any $r\in(0,r_0)$, where
	\begin{equation}\label{dfZmu}
	Z_{\mu}(\e)=\alpha_\mu(\e)\, V(\e),
	\end{equation}
	 $\psi_\mu(r)=\gamma(r)\mu(\mathcal{G}_r(0))$ and $\gamma(r)$ satisfies
	conditions \eqref{cndgammar}. Then the spectrum of the operator
	$H=-\Delta+V(\e)$ is discrete.
\end{theorem}

The following consequence of the previous theorem we obtain replacing the set 
$\mathrm{BP}_{eq}$ by its part $\mathrm{M}_f(\by,\,r)$ (see Definition \ref{dfMacry} and Proposition \ref{lmdescribedens}):
\begin{corollary}\label{thusecoredstort}
	Suppose that 
	 the condition 
	\begin{equation}\label{cndZmu3}
	\lim_{|\by|\rightarrow\infty}\;\sup_{\mu\in\mathrm{M}_f(\by,\,r)}\bar Z_{\mu}^\star(\psi_\mu(r),
	B_r(\by),\,\mu)=\infty
	\end{equation}
	is satisfied for some $r_0>0$ and any $r\in(0,r_0)$, where $Z_{\mu}$ is defined by \eqref{dfZmu},
 $\psi_\mu(\by,r)=\gamma(r)\mu(B_r(\by))$ and $\gamma(r)$ satisfies
	conditions \eqref{cndgammar}. Then the spectrum of the operator
	$H=-\Delta+V(\e)$ is discrete.
\end{corollary} 

On the base of 
Corollary \ref{thusecoredstort} we shall formulate a easier verifiable sufficient condition for discreteness of the spectrum of $H$ 
in terms of rearrangements of the potential $V(\e)$  on cubes of $m$-adic partitions of unit cubes with respect to Lebesgue measure. Recall that we have defined the cube $Q_r(\by)$ in Remark \ref{remoncubes}.
Consider a covering of the space $\R^d$ by the cubes $Q_1(\vec l)\;(\vec l\in\Z^d)$ and for any $\vec l\in\Z^d$ consider a sequence $\{D_j(\vec l)\}_{j=1}^\infty$ of 
subsets of $Q_1(\vec l)$. Furthermore, for some integers $n>0$ and $m>1$ consider the $m$-adic partition of each cube $Q_1(\vec l)$: $\{Q(\vec\xi,n)\}_{\vec\xi\in\Xi_n(\vec l)}$,
where $Q(\vec\xi,n)=Q_{m^{-n}}(\vec\xi)$ and $\Xi_n(\vec l)=\{\vec\xi\in m^{-n}\cdot\Z^d\,:\,Q(\vec\xi,n)\subset  Q_1(\vec l)\}$. Denote
\begin{equation}\label{dfXijvecl}
\Xi_n(\vec l,\,j)=\{\vec\xi\in m^{-n}\cdot\Z^d\,:\,Q(\vec\xi,n)\subseteq  D_j(\vec l)\}.
\end{equation} 
\begin{remark}\label{rembriefnot1}
For brevity we shall write in the next theorem and in its proof  $\bar Z^\star_{\mu}(u,\,\by\,,r)$, $\bar V^\star(u,\,\by,\,r)$ and $\bar V^\star(u,\,\vec\xi,\,n)$ instead of
\begin{equation*}
\bar Z^\star_{\mu}\big(u\cdot\mu(B_r(\by)),\,B_r(\by),\,\mu\big),\quad 
\bar V^\star\big(u\cdot
\mathrm{mes}_d(Q_{r}(\by)),\,Q_{r}(\by)\big) 
\end{equation*}
and $\bar V^\star\big(u\cdot\mathrm{mes}_d(Q(\vec\xi,n)),\,Q(\vec\xi,n)\big)$.	
\end{remark}

The promised theorem is following:

\begin{theorem}\label{thlogmdenspolyhed}
	Suppose that, $\theta\in(0,1)$ and for each $\vec l\in\Z^d$ the sequence $\{D_j(\vec l)\}_{j=1}^\infty$ forms a $(\log_m,\,\theta)$- dense system in $Q_1(\vec l)$. Furthermore,
	suppose that 
	\begin{equation}\label{Xinonempty}
	\forall\;\vec l\in\Z^d,\;n\in\N,\;j\in\{1,2.\dots,n\}:\quad \Xi_{n}(\vec l,j)\neq\emptyset. 
	\end{equation}
	Let $\gamma(r)$ be a nondecreasing monotone function 
	satisfying condition \eqref{cndgammar}. If for 
	any natural $n$ the condition 
	\begin{equation}\label{cndlogdenspolyhed}
	\lim_{|\vec l|\rightarrow\infty}\;\min_{\vec\xi\in\bigcup_{j=1}^n\Xi_{n}(\vec l,\,j)}\bar V^\star\big(\gamma(m^{-n}),\,\vec\xi,\,n\big)=\infty
	\end{equation}
	is satisfied, then the spectrum of the operator $H=-\Delta+V(\e)\cdot$ is discrete.
\end{theorem}
\begin{remark}\label{remonexampl}
	 We shall construct a family of potentials $V(\e)$ (see Section \ref{sec:examples}, Examples \ref{ex4}, \ref{excorpolyhed}) such that   conditions 
	 of Theorem \ref{thlogmdenspolyhed} are satisfied for them (hence the spectrum of $H$ is discrete), but them do not satisfy condition 
	 \eqref{cndGMD} of \cite{4}, and this family contains potentials which do not satisfy condition \eqref{cndSVy} of Theorem \ref{lmmeasinstcap1}.
\end{remark}

\section{Proof of main results} \label{sec:proofmainres}
\setcounter{equation}{0}

Recall that we use brief notations indicated in Remark \ref{rembriefnot}.

\subsection{Proof of Theorem \ref{thmeasinstcapcalG}}
\begin{proof}
 	Consider the following quantities connected 
	with the domain $\mathcal{G}$ having the form \eqref{formofcalG}:
	\begin{equation}\label{dfbarrrm}
	\bar r:=\max_{\omega\in S^{d-1}}r(\omega),\quad r_m:=\min_{\omega\in S^{d-1}}r(\omega),
	\end{equation}
	\begin{equation}\label{dfconstG}
	G:=\Big(\frac{\bar r}{r_m}\Big)^d.
	\end{equation}
	we see from \eqref{formofcalG}, \eqref{dfcalGry}, \eqref{dfbarrrm}
	and \eqref{dfconstG} that $\mathcal{G}_r(\by)\subseteq B_{\bar r\cdot r}(\by)$ 
	and 
	\begin{equation*}
	\mathrm{mes}_d(B_{\bar r\cdot r}(\by))\le G\cdot\mathrm{mes}_d(\mathcal{G}_r(\by)).
	\end{equation*}
	Let us define
	\begin{equation}\label{dfgamcalG}
	\gamma(r)=\big(\tilde\gamma(r)/G\big)^{(d-2)/d}.
	\end{equation}
	In view of \eqref{cndtildgam}. the function $\gamma(r)$ satisfies condition \eqref{cndgammar}.
	Suppose that $F\in\mathcal{N}_{\gamma(r)}(\by,r)$, where the collection $\mathcal{N}_{\gamma}(\by,r)$
	is defined by \eqref{defNcap}.  In view of
	the isocapacity inequality \eqref{isocapineq}
 	and the fact that it comes as
	identity for $F=\bar B_{\bar r\cdot r}(\by)$, we have, taking into account
	\eqref{dfgamcalG}:
	\begin{eqnarray*}\
	&&\frac{\mathrm{mes}_d(F)}{\mathrm{mes}_d
		(\mathcal{G}_r(\by))}\le G\frac{\mathrm{mes}_d(F)}{\mathrm{mes}_d
		(B_{\bar r\cdot r}(\by))}\le G\Big(\frac{\mathrm{cap}(F)}{\mathrm{cap}(\bar
		B_{\bar r\cdot r}(\by))}\Big)^{d/(d-2)}\le\nonumber\\
	&&  G\Big(\frac{\mathrm{cap}(F)}{\mathrm{cap}(\bar
		{\mathcal G}_r(\by))}\Big)^{d/(d-2)}\le G(\gamma(r))^{d/(d-2)}=\tilde\gamma(r),
	\end{eqnarray*}
	i.e., $F\in\mathcal{M}_{\tilde\gamma(r), \mathcal{G}}(\by,r)$. Thus,
	$\mathcal{N}_{\gamma(r)}(\by,r)\subseteq\mathcal{M}_{\tilde\gamma(r)}(\by,r)$.
	Hence
	\begin{equation*}
	\inf_{F\in\mathcal{N}_{\gamma(r)}(\by,r)}\int_{B_r(\by)\setminus
		F}V(\e)\, \mathrm{d}\e\ge
	\inf_{F\in\mathcal{M}_{\tilde\gamma(r)}(\by,r)}\int_{B_r(\by)\setminus
		F}V(\e)\, \mathrm{d}\e.
	\end{equation*}
	Therefore in view of condition \eqref{cndmeasinstcapcalG}, condition \eqref{cndmolch1} is satisfied. 
	Hence by Theorem \ref{thMazSh} the spectrum of the operator $H$ is discrete.
	
	Theorem \ref{thmeasinstcapcalG} is proven.
\end{proof}

\subsection{Proof of Theorem \ref{prsolextrprob}}
\begin{proof}
	In addition to the notations \eqref{dfIVOm}-\eqref{dfKVsminOm}  let us denote
	\begin{equation*}
	\mathcal{K}_W\big(t,\,X,\,\mu\big)={\mathcal L} _\star(s,X,\,\mu)\vert_{s=W_\star(t,\,X,\,\mu)},
	\end{equation*}
	\begin{equation*}
	\kappa_W(t,\,X,\,\mu)=\mu\big(\mathcal{K}_W(t,\,X,\,\mu)\big).
	\end{equation*}
	For the brevity we shall omit $W$, $X$ and $\mu$ in the brackets and in the subscripts of all the notations mentioned above.
	
	(i) It is easy to show that the function
	$\lambda_\star(s)$ is non-decreasing, right
	continuous, for any $s>0$ $\mathcal{L}_\star(s)\setminus\mathcal{L}_\star(s^-)=\{x\in X\,:\,W(x)=s\}$
	and $\mu\big(\mathcal{L}_\star(s)\setminus\mathcal{L}_\star(s^-)\big)=\lambda_\star(s)-\lambda_\star(s^-)$.
	Furthermore, in view of definition \eqref{dfstOm}, 
	\begin{equation*}
	\kappa^-(t)=\lim_{\,s\,\uparrow\,W_\star(t)}\lambda_\star(s)\le t\le\lim_{\,s\,\downarrow\,W_\star(t)}\lambda_\star(s)\le \kappa(t).
	\end{equation*}
By Proposition \ref{lmpartofset}, there exists a
	$\mu$-measurable set $D$ contained in ${\mathcal K}(t)\setminus{\mathcal K}^-(t)$  such that 
	$\mu(D)=t-\kappa^-(t)$, hence \eqref{estmesKV} is valid with
	$\tilde{\mathcal K}={\mathcal K}^-(t)\cup D$.
	Furthermore, it is clear that $W(x)<W_\star(t)$ for $x\in{\mathcal K}^-(t)$, $W(x)>W_\star(t)$ for $x\in X\setminus{\mathcal K}(t)$ and 
	$W(x)=W_\star(t)$ for $x\in{\mathcal K}(t)\setminus{\mathcal K}^-(t)$. These circumstances and \eqref{dfJVOm} imply that the representation \eqref{reprJV} is valid and the set $\tilde{\mathcal K}$ has the properties  \eqref{propKV}.
	
	(ii) In view of \eqref{reprJV}, \eqref{estmesKV} and definition \eqref{dfIVOm} 
	the inequality $It)\le
	J(t)$ is valid. Let us prove the inverse
	inequality. Using again definition \eqref{dfIVOm}, let us take an
	arbitrary $\epsilon>0$ and a set  $E\in\mathcal{E}(t,X,\,\mu)$, i.e., 
	\begin{equation}\label{estmesG}
\mu(E)\ge t,
\end{equation}	
	such that
	\begin{equation}\label{estintKV}
	\int_E W(x)\, \mu(\mathrm{d}x)\le I(t)+\epsilon.
	\end{equation}
We have:
	\begin{equation}\label{reprintbyG}
	\int_E W(x)\, \mu(\mathrm{d}x)=\int_{E\cap\tilde{\mathcal
			K}}W(x)\,
	\mu(\mathrm{d}x)+\int_{E\setminus\tilde{\mathcal
			K}}W(x)\, \mu(\mathrm{d}x).
	\end{equation}
	In view of \eqref{propKV}-b,
	\begin{equation}\label{estintbyGminusKV}
	\int_{E\setminus\tilde{\mathcal K}}W(x)\,
	\mu(\mathrm{d}x)\ge
	W_\star(t)\mu\big(E\setminus\tilde{\mathcal
		K}\big).
	\end{equation}
	On the other hand, in view of \eqref{estmesKV} and
	\eqref{estmesG},
	\begin{eqnarray*}
		&&\mu(E)=\mu\big(E\setminus\tilde{\mathcal
			K}\big)+\mu\big(E\cap\tilde{\mathcal
			K}\big)\ge\\
		&&\mu\big(\tilde{\mathcal K}\big)=
		\mu\big(\tilde{\mathcal K}\setminus
		E\big)+\mu\big(\tilde{\mathcal K}\cap E\big).
	\end{eqnarray*}
	Hence $\mu\big(E\setminus\tilde{\mathcal
		K}\big)\ge\mu\big(\tilde{\mathcal
		K}\setminus E\big)$ and we have in view of \eqref{estintKV}, \eqref{reprintbyG},
	\eqref{propKV}-a, \eqref{estintbyGminusKV} and \eqref{reprJV}:
	\begin{eqnarray*}
		&&I(t)+\epsilon\ge\int_E W(x)\, \mu(\mathrm{d}x)\ge\int_{E\cap\tilde{\mathcal
				K}}W(x)\,
		\mu(\mathrm{d}x)+W_\star(t)\mu\big(\tilde{\mathcal
			K}\setminus
		E\big)\ge\\
		&&\int_{E\cap\tilde{\mathcal K}}W(x)\,
		\mu(\mathrm{d}x+\int_{\tilde{\mathcal K}\setminus
			E}W(x)\, \mu(\mathrm{d}x)=\int_{\tilde{\mathcal
				K}}W(x)\, \mu(\mathrm{d}x=J(t).
	\end{eqnarray*}
	Since $\epsilon>0$ is arbitrary, from the last
	estimate follows inequality
	$I(t)\ge J(t)$. Thus,
	equality \eqref{IeqJ} is valid. Theorem \ref{prsolextrprob} is proven.
\end{proof}
In the proof of Theorem \ref{prsolextrprob} we have used the
following claim, which is   Sierpinski's 
theorem on non-atomic measures:
\begin{proposition}\cite{19}\label{lmpartofset}
	Suppose that conditions of Problem \ref{extrprobl} and Theorem
	\ref{prsolextrprob} are satisfied. Let $F\subseteq X$ be a $\mu$-measurable set such that
	$\mu(F)>0$. Then for any $t\in(0,\,\mu(F))$ there exists a
	$\mu$-measurable subset $F_t$ of $F$ such that
	$\mu(F_t)=t$.
	\end{proposition}
 
\subsection{Proof of Proposition \ref{lmJR}}

\begin{proof}
	Like above, we shall omit $W$, $X$ and $\mu$ in the brackets and in the subscripts of the notations \eqref{dfJVOm}-\eqref{dfKVsminOm} and \eqref{dfSVyrdel}-\eqref{dfLVsry}.
	By \eqref{dfSVyrdel}- \eqref{dfLVsry},
	\begin{equation}\label{choices0}
	\forall\,\epsilon>0\;\exists\,s_0>\bar W^\star(t)-\epsilon :\;\mu\big(\mathcal{L}^\star(s_0)\big)\ge t.
	\end{equation}
Denote  
	\begin{equation}\label{definsigmat}
	\sigma(t)=\mu(X)-t
	\end{equation}
	and $\tilde t=\frac{1}{\theta}t$.
	By claim (i) of Theorem \ref{prsolextrprob}, there exists a $\mu$-measurable set $\tilde{\mathcal K}\subseteq X$ such that
	\begin{eqnarray}\label{choiceK}
	J(\sigma(\tilde t))=\int_{\tilde{\mathcal K}}W(x)\,
	\mu(\mathrm{d}x),\quad 
	\mu(\tilde{\mathcal K})=\sigma(\tilde t).
	\end{eqnarray}
	Let us estimate:
	\begin{eqnarray*}
		&&\mu(X)+\mu\big(\tilde{\mathcal K}\cap \mathcal{L}^\star(s_0)\big)\ge\mu\big(\tilde{\mathcal K}\cup \mathcal{L}^\star(s_0)\big)+\mu\big(\tilde{\mathcal K}\cap \mathcal{L}^\star(s_0)\big)=\\
		&&\mu\big(\mathcal{L}^\star(s_0)\big)+\mu(\tilde{\mathcal K})\ge t+\sigma(\tilde t)=\mu(X\big)
		+(\theta-1)\tilde t,
	\end{eqnarray*}
	therefore
	\begin{equation*}
	\mu\big(\tilde{\mathcal K}\cap\mathcal{L}^\star(s_0)\big)\ge(\theta-1)\tilde t.
	\end{equation*}
	Then taking into account \eqref{choices0}, \eqref{dfLVsry}-b and the condition $W(x)\ge 0$, we have: 
	\begin{eqnarray*}
		&&	J(\sigma(\tilde t))=\int_{\tilde{\mathcal K}}W(x)\,
		\mu(\mathrm{d}x)\ge\int_{\tilde{\mathcal K}\cap
			{\mathcal L}^\star(s_0)}
		W(x)\,\mu(\mathrm{d}x)\ge\\
		&&(\theta-1)\tilde t\cdot s_0\ge(\theta-1)\tilde t(\bar W^\star(t)-\epsilon).
	\end{eqnarray*}
	Since $\epsilon>0$ is arbitrary, we get the desired estimate \eqref{estJR1}.
	
	Let us prove estimate \eqref{estJR2}. 
 	Using again claim (i) of Theorem \ref{prsolextrprob}, we can choose a $\mu$-measurable set $\tilde{\mathcal K}\subseteq X$ such that 
	\begin{equation}\label{mutiledK}
	\mu(\tilde{\mathcal K})=\sigma(t),
	\end{equation} 
\begin{eqnarray}\label{ineqforK}
	&&\forall\,x\in\tilde{\mathcal K}:\;W(x)\le W_\star(t),\quad 
	\forall\, x\in X\setminus\tilde{\mathcal K}:\;W(x)\ge W_\star(t)
	 \end{eqnarray} 
	and
	\begin{equation}\label{estabovJ}
	J(\sigma(t))=\int_{\tilde{\mathcal K}}W(x)\,
	\mu(\mathrm{d}x)\le W_\star(t)\mu(\tilde{\mathcal K})= W_\star(t)\sigma(t),
	\end{equation}
	 	On the other hand, by \eqref{ineqforK}-b and \eqref{dfLVsry}-b, 
	  	$X\setminus\tilde{\mathcal K}\subseteq\mathcal{L}^\star( W_\star(t))$.
	 Furthermore, in view of \eqref{mutiledK} and \eqref{definsigmat}, $\mu(X\setminus\tilde{\mathcal K})=t$. Hence by definition 
	\eqref{dfSVyrdel}, $W_\star(t)\le \bar W^\star(t)$. 
	This circumstance together with \eqref{estabovJ} imply the desired estimate \eqref{estJR2}. 
\end{proof}

\subsection{Proof of Theorem \ref{lmmeasinstcapnew}}

\begin{proof}
	It is known that the Lebesgue measure $\mathrm{mes}_d$ is non-atomic.
	Then by Theorem \ref{prsolextrprob} with $\mu=\mathrm{mes}_d$,
	$X=\mathcal{G}_r(\by)$ and $W(\e)\equiv V(\e)$, we obtain taking into account the inclusion
	\begin{equation*}
	\{E={\mathcal G}_r(\by)\setminus
	F:\;F\in\mathcal{M}_{\tilde\gamma(r)}(\by,r)\}\subset\mathcal{E}\big(\sigma(r),\,\mathcal{G}_r(\by),\,\mathrm{mes}_d\big),
	\end{equation*}
	that
	\begin{eqnarray*}
	&&\inf_{F\in\mathcal{M}_{\tilde\gamma(r)}(\by,r)}\int_{{\mathcal G}_r(\by)\setminus F} V(\e)\,
	\mathrm{d}\e\ge \inf_{E\in\mathcal{E}\big(\sigma(r),\,\mathcal{G}_r(\by),\,\mathrm{mes}_d\big)}\int_E V(\e)\,
	\mathrm{d}\e=\\
	&&J_V(\sigma(r),\,\by\,,r),
	\end{eqnarray*}
	where the collection $\mathcal{E}(t,X,\,\mu)$ is defined  in the formulation of Problem \ref{extrprobl} (in our case $\Sigma=\Sigma_L({\mathcal G}_r(\by))$). This estimate, condition \eqref{cndjVy} and Theorem \ref{thmeasinstcapcalG} imply the desired claim. Theorem \ref{lmmeasinstcapnew} is proven.
\end{proof}

\subsection{Proof of Theorem \ref{lmmeasinstcap1}}
\begin{proof}
Let us take $\theta>1$, $\tilde\gamma(r)=\hat\gamma(r)/\theta$ and $\sigma(r)=(1-\tilde\gamma(r))\mathrm{mes}_d(\mathcal{G}_r(0))$. Since  $\hat\gamma(r)$ satisfies conditions \eqref{cndtildgam}, then $\tilde\gamma(r)$ satisfies these conditions. 
In view of estimate \eqref{estJR1} (Proposition \ref{lmJR}) with $W(\e)=V(\e)$, $X=\mathcal{G}_r(\by)$ and
$t=\hat\delta(r)$, we have:
\begin{equation}\label{estJvfrbelow}
J_V(\sigma(r),\,\by\,,r)\ge\frac{\theta-1}{\theta}\hat\delta(r)\bar V^\star(\hat\delta(r),\by,r).
\end{equation}
This estimate, condition \eqref{cndSVy} and Theorem \ref{lmmeasinstcapnew} imply the desired claim.
Theorem  \ref{lmmeasinstcap1} is proven.
\end{proof}

\subsection{Proof of Theorem \ref{thusecore1}}

\begin{proof}
	By definition \eqref{defcore} of the base polyhedron for the harmonic
	capacity on $\bar\Omega=\bar{\mathcal G}_r(\by)$, for any
	$\mu\in\mathrm{BP}(\by,r)$ and any Borel set $F\subseteq\bar
	{\mathcal G}_r(\by)$
	\begin{equation*}
	\frac{\mu(F)}{\mu(\bar
		{\mathcal G}_r(\by))}\le\frac{\mathrm{cap}(F)}{\mathrm{cap}(\bar {\mathcal G}_r(\by))}.
	\end{equation*}
	Hence, in view of definitions \eqref{defNcap} and \eqref{defMmu},
	the inclusion  
	\begin{equation*}
	\mathcal{N}_{\gamma(r)}(\by,r)\subseteq
	\mathcal{M}^\mu_{\gamma(r)}(\by,r) 
	\end{equation*}
	is valid. This circumstance implies
	that
	\begin{eqnarray*}
		\inf_{F\in\mathcal{N}_{\gamma(r)}(\by,r)}\int_{{\mathcal G}_r(\by)\setminus
			F}V(\e)\, \mathrm{d}\e\ge
		\sup_{\mu\in\mathrm{\mathrm{BP}(\by,r)}}\;\inf_{F\in\mathcal{M}^\mu_{\gamma(r)}(\by,r)}\int_{{\mathcal G}_r(\by)\setminus
			F}V(\e)\, \mathrm{d}\e.
	\end{eqnarray*}
	The last estimate and condition \eqref{cndusecore4} imply that
	condition \eqref{cndmolch1} of  Theorem \ref{thMazSh} is
	fulfilled. Hence the spectrum of the operator $H$ is discrete. Theorem \ref{thusecore1} is proven.
\end{proof}

\subsection{Proof of Theorem \ref{thusecore2}}

\begin{proof}
Let us take $\theta>1$ and denote 
\begin{equation*}
\tilde\gamma(r)=\gamma(r)/\theta,\quad \sigma(r)=(1-\tilde\gamma(r))\mu(\mathcal{G}_r(0)).
\end{equation*}
  We have, taking into account the 
inclusions $\mathrm{BP}_{eq}(\by,r)\subseteq\mathrm{BP}(\by,r)$,
\begin{equation*}
\{E={\mathcal G}_r(\by)\setminus
F:\;F\in\mathcal{M}_{\tilde\gamma(r)}^\mu(\by,r)\}\subseteq\mathcal{E}\big(\sigma(r),\,\mathcal{G}_r(\by),\,\mu\big),
\end{equation*}
that
	\begin{eqnarray}\label{estmuBP}
		&&\sup_{\mu\,\in\,
			\mathrm{BP}(\by,r)}\;\inf_{F\in\mathcal{M}^\mu_{\tilde\gamma(r)}(\by,r)}\int_{\mathcal{G}_r(\by)\setminus
			F}V(\e)\,\mathrm{d}\e\ge\\
		&&\sup_{\mu\,\in\,
			\mathrm{BP}_{eq}(\by,r)}\;\inf_{F\in\mathcal{M}^\mu_{\tilde\gamma(r)}(\by,r)}\int_{\mathcal{G}_r(\by)\setminus
			F}Z_\mu(\e)\,\mu(\mathrm{d}\e)\ge\nonumber\\
		&&\sup_{\mu\,\in\,
			\mathrm{BP}_{eq}(\by,r)}\;\inf_{E\in\mathcal{E}\big(\sigma(r),\,\mathcal{G}_r(\by),\,\mu\big)}\int_E Z_\mu(\e)\,\mu(\mathrm{d}\e).\nonumber
	\end{eqnarray}
	Recall that the collection $\mathcal{E}(t,X,\,\mu)$ is defined  in the formulation of Problem \ref{extrprobl} (in our case $\Sigma=\Sigma_L({\mathcal G}_r(\by))$) and the function $Z_\mu(\e)$ is defined by \eqref{dfZmu},
	\eqref{dfalphamu}. Since $V\in L_{1,\,loc}(\R^d)$, $Z_\mu(\e)$  belongs to $L_1(\mathcal{G}_r(\by),\,\mu)$ for any $\mu\,\in\,
	\mathrm{BP}_{eq}(\by,r)$. 
	Hence since the Lebesgue measure $\mathrm{mes}_d$ is is non-atomic and $\sigma$-finite, and each measure from $\mathrm{BP}_{eq}(\by,r)$ is absolute continuous with respect to $\mathrm{mes}_d$, then $\mathrm{BP}_{eq}(\by,r)$ consists 
	of non-atomic measures (\cite{8}, Theorem 2,4). Furthermore, in view of definition \eqref{defcore} and boundedness of the domain $\mathcal{G}_r(\by)$, each measure from $\mathrm{BP}_{eq}(\by,r)$ is finite.  Hence using Theorem \ref{prsolextrprob} and estimate \eqref{estJR1} (Proposition \ref{lmJR}), we obtain:
	\begin{eqnarray*}
	&&\hskip-5mm\sup_{\mu\,\in\,
		\mathrm{BP}_{eq}(\by,r)}\;\inf_{E\in\mathcal{E}\big(\sigma(r),\,\mathcal{G}_r(\by),\,\mu\big)}\int_E Z_\mu(\e)\,\mu(\mathrm{d}\e)=\\
	&&\hskip-5mm\sup_{\mu\,\in\,
		\mathrm{BP}_{eq}(\by,r)}J_{Z_\mu}(\sigma(r),\by,r,\mu)\ge
\frac{\theta-1}{\theta}\psi_\mu(r)\sup_{\mu\,\in\,
	\mathrm{BP}_{eq}(\by,r)}\bar Z_\mu^\star(\psi_\mu(r),\,\by,\,r,\,\mu).
	\end{eqnarray*}
	The last estimate,
	\eqref{estmuBP}, condition \eqref{cndZmu1} and Theorem \ref{thusecore1}  imply the desired claim.	Theorem \ref{thusecore2} is proven.
\end{proof}

\subsection{Proof of Corollary \ref{thusecoredstort}} 
\begin{proof}
	The inclusion  $M_f(\by,r)\subseteq BP_{eq}(\by,r)$ (Proposition \ref{lmdescribedens}) implies:
	\begin{equation*}
	\sup_{\mu\,\in\,\mathrm{BP}_{eq}(\by,r)}\bar Z_\mu^\star(\psi_\mu(r),\,\by,\,r,\,\mu)\ge
	\sup_{\mu\,\in\,M_f(\by,r)}\bar Z_\mu^\star(\psi_\mu(r),\,\by,\,r,\,\mu).
	\end{equation*}
	This estimate, condition \eqref{cndZmu3} and Theorem \ref{thusecore2} imply the desired claim.
\end{proof}

\subsection{Proof of Theorem \ref{thlogmdenspolyhed}}
\begin{proof}
	In this proof and in the lemmas, applied in it, we shall use the brief notations indicated in Remark \ref{rembriefnot1}.
	Let us take a ball $B_r(\by)$ and consider on it the probability measure $m_{d,r}$, defined by 
	\eqref{dfmdr}. 
	 Consider the function $s_{r,\by}:\,B_r(\by)\rightarrow [0,1]$, defined in the following manner:
	\begin{equation}\label{dfsrbynew}
	s_{r,\by}(\e):=m_{d,r}\big(\{\bs\in B_r(\by):\,P_1\bs\le P_1\e\}\big),
	\end{equation}
	where $P_1$ is the
	following operator $P_1:\,\R^d\rightarrow\R$:
	\begin{equation}\label{dfP1}
	\mathrm{for}\quad
	\e=(x_1,x_2,\dots,x_d)\quad P_1\e:=x_1.
	\end{equation}
	It is easy to check, that $s_{r,\by}$ is a measure preserving mapping (Definition \ref{dfmeaspres}).
	Denote by $\by_r^-$ the left point of the two-point set 
	$\partial(B_r(\by))\cap\big((I-\tilde P_1)^{-1}(I-\tilde P_1)\by\big)$, where $\tilde P_1$ is 
	the orthogonal projection in $\R^d$ on the first coordinate axis, i.e., for $\e=(x_1,x_2,\dots,x_d)$ $\tilde P_1\e:=(x_1,0,\dots,0)$.
	Consider the function $\bz=F_r(\e):=\frac{\e-\by_r^-}{r}$, which maps bijectively the ball $B_r(\by)$ onto the ball $B_1(\mathbf{e})$ 
	with $\mathbf{e}=(1,0,\dots,0)$. It is easy to see that
	\begin{equation}\label{propsrbynew}
	s_{r,\by}(\e)=s_{1,\mathbf{e}}(F_r(\e)).
	\end{equation}
	Consider the function $f(t)=t^{(d-2)/d}$ and the absolutely continuous measure $\mu_s$ on $B_r(\by)$, whose density is $f^\prime\circ s_{r,\by}$, i.e., 
	this measure belongs to $M_f(r,\by)$ (Definition \ref{dfMacry}). This means that for any  set $A\in \Sigma_L(\bar B_r(\by))$
	\begin{equation}\label{dfmusnew}
	\mu_s(A)=\mathrm{cap}(B_r(0))\int_Af^\prime(s_{r,\by}(\e))m_{d,r}(\mathrm{d}\e),
	\end{equation}
	After change of the variable $\bz=F_r(\e)$ we get, taking into account \eqref{propsrbynew}:
	\begin{equation}\label{dfmus1new}
	\mu_s(A)=\mathrm{cap}(B_r(0))\int_{F_r(A)}f^\prime(s_{1,\mathbf{e}}(\bz))m_{d,1}(\mathrm{d}\bz).
	\end{equation}
	Consider the function $Z_{\mu_s}(\e)$, defined by \eqref{dfZmu} and \eqref{dfalphamu} with 
	$\mu=\mu_s$, i.e.,
	\begin{eqnarray}\label{dfZmusnew}
\hskip10mm	Z_{\mu_s}(\e)=V(\e)\big(f^\prime(s_{r,\by}(\e))\big)^{-1}=
	d/(d-2)V(\e)\big(s_{r,\by}(\e)\big)^{2/d}.
	\end{eqnarray}
	 Using Lemma \ref{lmineqforrear} with $t=\hat\gamma(\tilde r,K)\;(\tilde r=r/\sqrt{d})$,
	where
	\begin{equation}\label{dfhatgamrKnew}
	\hat\gamma(\rho,K)=K\gamma(\theta \rho/m^2)\theta^d\quad (K>0).
	\end{equation}
	and Lemma \ref{prestbarVstarbelow}, 
 taking in these lemmas $W(\e)=V(\e)$, we get that for some $\kappa,\delta\in(0,1)$, $K>0$ and any $\by\in\R^d$, $r\in(0,1)$ there 
	are $\vec l(\by,r)\in\Z^d$, a cube $Q_{\tilde r}(\tilde\by)\subseteq B_r(\by)\cap Q_1(\vec l(\by,r))$ and $j\in\{1,2,\dots,n\}$ with
	$n=\big[\log_m\big(\frac{1}{\theta\tilde r}\big)\big]+2$ such that 
	\begin{eqnarray}\label{ineqforrearV1}
	&&\bar Z^\star_{\mu_s}\big(\bar\gamma(r),\,\by,\,r\big)\ge\\
	&&\delta\cdot\bar V^\star\big(\hat\gamma(\tilde r,K),\,\tilde\by,\,\tilde r
	\big)\ge\delta\min_{\vec\xi:\,Q(\vec\xi,\,n)\subseteq F_j}\bar V^\star\big(\gamma(m^{-n}),\,\vec\xi\,,n),\nonumber
	\end{eqnarray}
	where $\bar\gamma(r)=\kappa\hat\gamma(r/\sqrt{d},\,K)$ and $F_j$ is a non-empty union of cubes $Q(\vec\xi,\,n)$ such that $F_j\subseteq Q_{\bar r}(\by)\cap D_j(\vec l(\by,r))$. Notice that since the function $\gamma(r)$ satisfies conditions \eqref{cndgammar}, the function $\bar\gamma(r)$ satisfies this condition too for some $r_0>0$.
	Then in view of condition \eqref{Xinonempty}, estimate \eqref{ineqforrearV1}  implies:
$\bar Z^\star_{\mu_s}\big(\bar\gamma(r),\,\by,\,r\big)\ge
		\delta \min_{\vec\xi\,\in\bigcup_{j=1}^n\Xi_{n}(\vec l,\,j)}\bar V^\star\big(\gamma(m^{-n}),\,\vec\xi\,,n)$.
  Since 
	\begin{equation*}
	B_r(\by)\cap Q_1(\vec l(\by,r))\neq\emptyset  
	\end{equation*}
	for any $\by\in\R^d$,
	 this estimate, inclusion $\mu_s\in M_f(\by,r)$ and 
	condition \eqref{cndlogdenspolyhed} imply that condition \eqref{cndZmu3}
	of Corollary \ref{thusecoredstort} is satisfied with $\gamma(r)=\bar\gamma(r)$. 
	Then the spectrum of the operator $H=-\Delta+V(\e)\cdot$ is discrete. Theorem \ref{thlogmdenspolyhed} is proven. 
\end{proof}

In the proof of Theorem \ref{thlogmdenspolyhed} we have used the following claims:

\begin{lemma}\label{lmineqforrear}
	For some $\kappa,\delta\in(0,1)$ and any ball $B_r(\by)$ with $r\in(0,1)$ there are $\vec l=\vec l(\by,r)\in\Z^d$ and a cube 
	\begin{equation}\label{incltointersect}
	Q_{\tilde r}(\tilde\by)\subseteq B_r(\by)\cap Q_1(\vec l(\by,r))
	\end{equation}
	with $\tilde r=r/\sqrt{d}$ such that for any nonnegative function 
	$W\in L_1(B_r(\by))$ and $t\in(0,1)$ the inequality
	\begin{equation}\label{ineqforrear}
	\bar Z^\star_{\mu_s}\big(\kappa t,\,\by,r\big)\ge\delta\cdot\bar W^\star\big(t,\,\tilde\by,\tilde r\big)
	\end{equation}
	is valid, where $\mu_s$ is the measure on $B_r(\by)$, defined by \eqref{dfmusnew} with $f(t)=t^{(d-2)/d}$ and $Z_{\mu_s}(\e)=W(\e)\big(f^\prime(s_{r,\by}(\e))\big)^{-1}$
	with the function $s_{r,\by}(\e)$, defined by \eqref{dfsrbynew}.
\end{lemma}
\begin{proof}
	Let us take $\delta>0$ and consider the set
	\begin{eqnarray}\label{dfPidelt}
&&	\Pi_\delta(r,\by):=\{\e\in B_r(\by):\\ &&(f^\prime(s_{r,\by}(\e)))^{-1}=d/(d-2)\big(s_{r,\by}(\e)\big)^{2/d}\ge\delta\},\nonumber
	\end{eqnarray}
	or in view of \eqref{propsrbynew},
	\begin{eqnarray*}
		&&\Pi_\delta(r,\by)=\{\e\in B_r(\by):\,P_1(\e-\by_r^-)\ge\sigma(\delta) r\},\nonumber\\
	&&\mathrm{where}\quad 
	\sigma(\delta)=P_1\big(s_{1,\mathrm{e}}^{-1}\big((d-2)(\delta/d)^{d/2}\big)\big).
	\end{eqnarray*}
By \eqref{dfZmusnew} and \eqref{dfPidelt}, we have for $N>0$:
	\begin{equation}\label{inclsets}
	\{\e\in B_r(\by):\,V(\e)\ge N/\delta \}\cap\Pi_\delta(r,\by)\subseteq\{\e\in B_r(\by):\,Z_{\mu_s}(\e)\ge N\}.
	\end{equation}
	The numbers $\delta$ and $N$ will be chosen below. Consider the cube $Q_{r_1}(\by_1)$ inscribed in the ball $B_r(\by)$, i.e., $r_1=2r/\sqrt{d}$ and $\by_1=\by-r\vec a$ with $\vec a=(d^{-1/2}, d^{-1/2},\dots,d^{-1/2})$.
	In what follows we shall choose $\delta>0$ such that
	\begin{equation}\label{inclQr1}
	Q_{r_1}(\by_1)\subset\Pi_\delta(r,\by).
	\end{equation}
	This condition is equivalent to $\sigma(\delta)\le 1-d^{-1/2}$. 
Since $r\in(0,1)$, the cube $Q_{r_1}(\by_1)$ intersects 
	not more than $2^d$ adjacent cubes $Q_1(\vec l_k)\,(\vec l_k\in\Z^d)$  and among them
	there is a cube $Q_1(\vec l_k)$ such that for some $\tilde\by\in Q_1(\vec l_k)$ $Q_{\tilde r}(\tilde\by)\subseteq Q_1(\vec l_k)\cap Q_{r_1}(\by_1)$
	with $\tilde r=r_1/2=r/\sqrt{d}$. Thus, since $Q_{r_1}(\by_1)\subset B_r(\by)$
	 for any $\by\in\R^d$ and $r\in(0,1)$ we can choose $\vec l(\by,r)\in Z^d$ such that
	inclusion \eqref{incltointersect} is valid.
Let us take $M=\bar W^\star\big(t,\tilde\by,r\big)$ for $t\in(0,1)$. Then by definitions \eqref{dfSVyrdel}-\eqref{dfLVsry},
	for any $\epsilon>0$ there is $s\in(M-\epsilon,\,M]$ such that $\mathrm{mes}_d\big(\{\e\in Q_{\tilde r}(\tilde\by):\,W(\e)\ge s\}\big)\ge t\cdot\mathrm{mes}_d(Q_{\tilde r}(0)$. Hence
	\begin{equation*}
	\mathrm{mes}_d\big(\{\e\in Q_{\tilde r}(\tilde\by):\,W(\e)\ge M-\epsilon\}\big)\ge\\
	t\cdot\mathrm{mes}_d(Q_{\tilde r}(0).   
	\end{equation*}
	Notice that in view of \eqref{inclQr1} and the inclusion 
	$Q_{\tilde r}(\tilde\by)\subseteq Q_{r_1}(\by_1)$, the inclusion $Q_{\tilde r}(\tilde\by)\subset\Pi_\delta(r,\by)$ is valid.
	Then inclusion \eqref{inclsets} with $N=\delta(M-\epsilon)$,   definitions \eqref{dfmusnew}, \eqref{dfPidelt} and equality 
	$\min_{t\in[0,1]}f^\prime(t)=(d-2)/d$ imply that for $t\in(0,1)$
	\begin{eqnarray*}
		&&\mu_s\big(\{\e\in B_r(\by):\,Z_{\mu_s,\,W}(\e)\ge \delta(M-\epsilon)\}\big)\ge\\
		&&\frac{d-2}{d}\frac{\mathrm{cap}(B_r(0))}{\mathrm{mes}_d(B_r(0))}
		\mathrm{mes}_d\big(\{\e\in Q_{\tilde r}(\tilde\by):\,W(\e)\ge M-\epsilon\}\big)\ge\\
		&&\frac{d-2}{d}\mathrm{cap}(B_r(0))t\cdot m_{d,r}(Q_{\tilde r}(\tilde\by))\ge
		\delta\frac{d-2}{d}t\cdot\mu_s(Q_{\tilde r}(\tilde\by))=\\
		&&\delta q\frac{d-2}{d}t\cdot\mu_s(B_r(\by)),
	\end{eqnarray*}
	where, in view of \eqref{dfmus1new},
	the quantity $q=\frac{\mu_s(Q_{\tilde r}(\tilde\by)}{\mu_s(B_r(\by))}$ does
	not depend on $r$ and $\by$. Denote $\kappa=\delta q\frac{d-2}{d}$. Then the last estimate imply 
	that  $\bar Z^\star_{\mu_s}\big(\kappa t,\,\by,\,r\big)\ge \delta(M-\epsilon)$.
	Since $\epsilon>0$ is arbitrary, we obtain the desired inequality \eqref{ineqforrear}.
	\end{proof}

\begin{lemma}\label{prestbarVstarbelow}
	Suppose that in a cube $Q_1(\vec l)\,(\vec l\in\Z^d)$ there is is a sequence of  subsets $\{D_n(\vec l)\}_{n=1}^\infty$ forming in it a $(\log_m,\,\theta)$-dense system.
	Let $\gamma:\,(0,\,r_0)\rightarrow\R$ be a monotone nondecreasing function with $r_0=\min\{1,\,1/(m^2\theta)\}$. Then 
for some  $K>0$ and for
	 any cube $Q_r(\by)\subset Q_1(\vec l)$ there are $j\in\{1,2,\dots,n\}$ with
	\begin{equation}\label{dfnlog1}
	n=\big[\log_m\big(\frac{1}{\theta r}\big)\big]+2
	\end{equation}
	and a non-empty set $F_j\subseteq Q_r(\by)\cap D_j(\vec l)$, which is a union of cubes 
	\begin{equation*}
	Q(\vec\xi\,,n)\;(\vec\xi\in m^{-n}\cdot\Z^d),
	\end{equation*}
	such that for any nonnegative function $W\in L_1(Q_r(\by))$ the inequality
	\begin{eqnarray}\label{ineqWstaryrxin}
	\bar W^\star(\hat\gamma(r,K),\,\by,\,r)\ge
	\min_{\vec\xi:\;Q(\vec\xi,\,n)\subseteq F_j}\bar W^\star\big(\gamma(m^{-n}),\,\vec\xi,\,n)\big) 
	\end{eqnarray}
	is valid, where the function $\hat\gamma(r,K)$ is defined by \eqref{dfhatgamrKnew}.
\end{lemma}
\begin{proof}
	Let us take $r\in (0,\,r_0)$. 
	 Then by Definition \ref{dfdenslogmtetpart1}, there is
	\begin{equation*}
	j\in\{1,2,\dots,n\} 
	\end{equation*} 
	 such that for some regular parallelepiped $\Pi\subseteq D_j(\vec l)$ there is a cube 
\begin{equation}\label{inclcube1}
	Q_{\theta r}(\bs)\subseteq\Pi\cap Q_r(\by).
	\end{equation}
	Definition \eqref{dfnlog1} implies that 
	\begin{equation}\label{ineqforn1}
	m^{-(n-1)}<\theta r\le m^{-(n-2)}.
	\end{equation}
 Using Lemma \ref{lminclrear2} and inclusion \eqref{inclcube1} we get:
	\begin{eqnarray}\label{estforbarV1}
	 \bar W^\star(\hat\gamma(r,K),\,\by,\,r)\ge
	\bar W^\star(K\gamma(\theta r/m^2),\,\bs,\,\theta r)
	\end{eqnarray}
	On the other hand, in view of \eqref{inclcube1} and the left inequality of \eqref{ineqforn1}, the set $\{\vec\xi:\,Q(\vec\xi,n)\subseteq Q_{\theta r}(\bs)\}$ 
	is not empty. Denote
	\begin{eqnarray}\label{dfFj1}
	F_j:=\bigcup_{\vec\xi:\,Q(\vec\xi,n)\subseteq Q_{\theta r}(\bs)}Q(\vec\xi,n).
	\end{eqnarray}
Then the set $Q_{\theta r}(\bs)\setminus F_j$ is a union of $2d$ regular parallelepipeds, each 
	of them is the cartesian product of a face of the cube $Q_{\theta r}(\bs)$ and an interval, whose length is less than $m^{-n}$. 
	Hence, in view of the right inequality \eqref{ineqforn1}, 
	\begin{equation*}
	\mathrm{mes}_d(Q_{\theta r}(\bs))-\mathrm{mes}_d(F_j)<2d\Big(m^{-n+2}\Big)^{d-1}m^{-n}.
	\end{equation*}
	Since the set $F_j$ is not empty, then in view of \eqref{dfFj1}, $\mathrm{mes}_d(F_j)\ge m^{-nd}$. Hence
	\begin{eqnarray*}
		\frac{\mathrm{mes}_d(Q_{\theta r}(\bs))}{\mathrm{mes}_d(F_j)}<\frac{2d\Big(m^{-n+2}\Big)^{d-1}m^{-n}}{ m^{-nd}}+1=2d\cdot m^{2(d-1)}+1.
	\end{eqnarray*}
	Notice that the function $\bar W^\star(t,\Omega)$ is non-increasing by $t$.
	Continuing estimate \eqref{estforbarV1} and using Lemmas \ref{lmsumrear1}, \ref{lminclrear2}, definition \eqref{dfFj1} and the right inequality \eqref{ineqforn1},  we obtain:taking 
	$K=\big(2d\cdot m^{2(d-1)}+1\big)^{-1}$:
	\begin{eqnarray*}\label{finest1}	
	 && \bar W^\star(\hat\gamma(r,K),\,\by,\,r)\ge\bar W^\star(\gamma(m^{-n})\mathrm{mes}_d(F_j),\,F_j)\ge\\ 
	 && \min_{\vec\xi:\;Q(\vec\xi,\,n)\subseteq F_j}\bar W^\star\big(\gamma(m^{-n}),\,\vec\xi,\,n\big),
 	\end{eqnarray*}
	i.e., the desired inequality \eqref{ineqWstaryrxin} is proven.
\end{proof}

In the proof of Lemma \ref{prestbarVstarbelow} we have used the following lemmas:

\begin{lemma}\label{lmsumrear1} 
	Let $\Omega\subseteq\R^d$ be a measurable set having the form $\Omega=\bigcup_{n=1}^N\Omega_n$, where $\Omega_n$ are measurable sets such that 
	$\mathrm{mes}_d(\Omega_k\cap\Omega_l)=0$ for $k\neq l$. If $W(\e)$ is a non-negative measurable function defined in $\Omega$,
	then the inequality is valid for $t\in(0,1)$:
	\begin{equation}\label{ineqrearrangonun1}
	\bar W^\star(t\cdot\mathrm{mes}_d(\Omega),\,\Omega)\ge\min_{1\le n\le N}\bar W^\star(t\cdot\mathrm{mes}_d(\Omega_n),\,\Omega_n). 
	\end{equation}
\end{lemma}
\begin{proof}
	By definition \eqref{dfSVyrdel}-\eqref{dfLVsry} of the non-increasing rearrangement $\bar W^\star$, for any $n\in\{1,2,\dots,N\}$ and $\epsilon>0$ there is $s_n>0$ such that 
	\begin{equation*}
	s_n>\bar W^\star(t\cdot\mathrm{mes}_d(\Omega_n),\,\Omega_n)-\epsilon
	\end{equation*}
	and $\lambda^\star(s_n,W,\Omega_n)\ge t\cdot\mathrm{mes}_d(\Omega_n)$.
	Let us take $s_0=\min_{1\le n\le N}s_n$. Then taking into account that the functions $\lambda^\star(s,W,\Omega_n)$ are non-increasing, we get:
	\begin{eqnarray*}
		&&	\lambda^\star(s_0,W,\Omega)=\sum_{n=1}^N\lambda^\star(s_0,W,\Omega_n)\ge \sum_{n=1}^N\lambda^\star(s_n,W,\Omega_n)\ge\\
		&&t\sum_{n=1}^N\mathrm{mes}_d(\Omega_n)=
		 t\cdot\mathrm{mes}_d(\Omega).
	\end{eqnarray*}
	Hence
	\begin{equation*}
	\bar W^\star(t\cdot\mathrm{mes}_d(\Omega),\,\Omega)\ge s_0>\min_{1\le n\le N}\bar W^\star(t\cdot\mathrm{mes}_d(\Omega_n),\,\Omega_n)-\epsilon.
	\end{equation*}
	Since $\epsilon>0$ is arbitrary, we get the desired inequality \eqref{ineqrearrangonun1}.
\end{proof}
\begin{lemma}\label{lminclrear2}
	Let $\Omega_1$ and $\Omega_2$ be measurable subsets of $\R^d$ such that $\Omega_1\subseteq\Omega_2$ and $W(\e)$ be a non-negative measurable function 
	defined on $\Omega_2$. Then for any $t>0$ the inequality
$\bar W^\star(t,\Omega_1)\le\bar W^\star(t,\Omega_2)$
is valid.
\end{lemma}
\begin{proof}
	In view of the inclusion $\Omega_1\subseteq\Omega_2$ and definition \eqref{dfLVsry}, $\lambda^\star(s,W,\Omega_1)\le\lambda^\star(s,W,\Omega_2)$.
	Hence 
	\begin{equation*}
	\{s>0:\,\lambda^\star(s,W,\Omega_1)\ge t\}\subseteq\{s>0:\,\lambda^\star(s,W,\Omega_2)\ge t\}.
	\end{equation*}
	This inclusion and definition  \eqref{dfSVyrdel} imply the desired claim.    
\end{proof}

\section{Some examples} \label{sec:examples}
\setcounter{equation}{0}
First of all, consider some examples of the $(\log_m,\,\theta)$-dense system (Definition \ref{dfdenslogmtetpart1}).

\begin{example}\label{ex1}
	Consider the classical middle third Cantor set $\mathcal{C}\subset[0,1]$, Let $I_{n,k}\;(n=1,2,\dots),\,k=1,2,\dots, 2^{n-1}$ be the closures 
	of intervals adjacent to $\mathcal{C}$. It is known that they are disjoint and for any fixed $n$ and each $k\in\{1,2,\dots, 2^{n-1}\}$ 
	$\mathrm{mes}_1(I_{n,k})=3^{-n}$. For fixed $n$ we shall number the intervals $I_{n,k}$ from the left to the right. Denote $D_n=\bigcup_{k=1}^{2^{n-1}}I_{n,k}$. Let us show that the sequence $\{D_n\}_{n=1}^\infty$ forms 
	in $[0,1]$ a $(\log_3,\,1/9)$-dense system. Let us take $\theta=1/9$, $r\in\big(0,\,\min\{1,\,1/(3^2\theta)\}\big)=(0,\,1)$, an interval $Q_r(y)=[y,\,y+r]$ and the natural number 
	$n=\Big[\log_3\Big(\frac{1}{r\theta}\Big)\Big]$. Then $n\ge 2$ and
	\begin{equation}\label{limitsforrtht}
	3^{-(n+1)}<r\theta\le 3^{-n}.
	\end{equation}
	Consider two cases:
	
	a) there are $j\in\{1,2,\dots,n-1\}$, $k\in\{1,2,\dots,2^{j-1}\}$ such that
	\begin{equation*}
	\mathrm{mes}_1\big(Q_r(y)\cap I_{j,k}\big)>3^{-n};
	\end{equation*}
	 
		b) for all $j\in\{1,2,\dots,n-1\}$, $k\in\{1,2,\dots,2^{j-1}\}$ $\mathrm{mes}_1\big(Q_r(y)\cap I_{j,k}\big)\le 3^{-n}$.
	
	In the case a)  in view of the right inequality \eqref{limitsforrtht}, $\mathrm{mes}_1\big(Q_r(y)\cap I_{j,k}\big)>r\theta$, hence there is a real $z$ 
	such that $Q_{r\theta}(z)\subset Q_r(y)\cap I_{j,k}$.
	
	Consider the case b). Let $(c_{n,k},\,d_{n,k})\;(k=1,2,\dots,2^{n-1})$ be the intervals forming the set $(0,\,1)\setminus\bigcup_{j=1}^{n-1}D_j$. 
	By the construction of Cantor set $\mathcal{C}$, 
$d_{n,k}-c_{n,k}=3^{-(n-1)}$
	and each interval $(c_{n,k},\,d_{n,k})$ contains an unique interval $I_{n,k}=[c_{n,k}+3^{-n},\,d_{n,k}-3^{-n}]\subset D_n$. 
	Assumption b) implies that for some $k\in\{1,2,\dots,2^{n-1}\}$ or the left edge $y$ of the interval $Q_r(y=[y,\,y+r]$ belongs 
	to $[c_{n,k}-3^{-n},\,c_{n,k}]$, or the right its edge $y+r$ belongs to $[d_{n,k},\,d_{n,k}+3^{-n}]$. 
	On the other hand, the left inequality of \eqref{limitsforrtht} (with $\theta=1/9$) implies that $r>3^{-(n-1)}$. All the above arguments 
	imply that $I_{n,k}\subseteq Q_r(y)$, hence in view of the right inequality \eqref{limitsforrtht}, 
	$\mathrm{mes}_1\big(Q_r(y)\cap I_{n,k}\big)=3^{-n}\ge r\theta$. Therefore in the case b) there is a real $z$ and $k\in\{1,2,\dots, 2^{n-1}\}$
	such that $Q_{r\theta}(z)\subset Q_r(y)\cap I_{n,k}$. Thus, the sequence $\{D_n\}_{n=1}^\infty$ satisfies the conditions of 
	Definition \ref{dfdenslogmtetpart1} with $d=1$, $m=3$, $\theta=1/9$ and $\Pi=I_{j,k}$.
\end{example}

\begin{example}\label{ex2}
	Consider a cube $Q_1(\by)\subset\R^d$, represented in the form $Q_1(\by)=Q_1(\by_1)\times Q_1(\by_2)$, where 
	$Q_1(\by_1)\subset\R^{d_1}$ and $Q_1(\by_2)\subset\R^{d_2}$. Let $\{D_n\}_{n=1}^\infty$ be a sequence of subsets of 
	the cube $Q_1(\by_1)$, forming in it a $(\log_m,\theta)$-dense system. It is easy to see that the sequence
	$\{D_n\times Q_1(\by_2)\}_{n=1}^\infty$ forms in $Q_1(\by)$ a $(\log_m,\theta)$-dense system too.	
\end{example}

\begin{example}\label{ex3}
	Let $\{D_n^{(i)}\}_{n=1}^\infty\;(i=1,2,\dots,I)$ be a finite collection of sequences of sets such that each of them forms a 
	$(\log_m,\theta)$-dense system $(m>1,\,\theta\in(0,\,1))$ in a cube $Q_1(\by_i)\subset\R^{d_i}$.  Consider the following 
	sequence of subsets of the cube $Q_1(\by)=\times_{i=1}^IQ_1(\by_i)\subset\R^d\;\big(\by=(\by_1,\by_2,\dots,\by_I),\,d=\sum_{i=1}^Id_i\big)$:
	\begin{equation}\label{dfmathcalD}
	\mathcal{D}_N=\bigcup_{\vec n\in\mathcal{E}_N}\times_{i=1}^ID_{n_i}^{(i)},
	\end{equation}
	where 
	\begin{equation}\label{dfcalEN}
	\mathcal{E}_N=\{\vec n=(n_1,n_2,\dots,n_I)\in\N^I:\,\max_{1\le i\le I}n_i=N\}
	\end{equation}
	Let us show that the sequence $\{\mathcal{D}_N\}_{N=1}^\infty$ forms a 
	$(\log_m,\theta)$-dense system in the cube $Q_1(\by)$. It is clear that for this sequence the condition (a) of Definition \ref{dfdenslogmtetpart1} is satisfied.
	Let us show that also condition (b) of this definition is satisfied for it. Let us take a cube $Q_r(\bz)=\times_{i+1}^IQ_r(\bz_i)\subset Q_1(\by)$ with $r\in\big(0,\,\min\{1,\,\frac{1}{\theta m^2}\}\big)$.
	By the condition (b) of the above mentioned  definition applied to  $i$-th sequence of sets, 
	  there is  
$j_i\in\{1,2,\dots, n\}$ with $n=\big[\log_m\big(\frac{1}{\theta r}\big)\big]$
such that for some regular parallelepiped $\Pi_i\subseteq D_{j_i}^{(i)}$ 
	 there is a cube $Q_{\theta r}(\bs_i)$, contained in $\Pi_i\cap Q_r(\bz_i)$. Denote 
	$Q_{\theta r}(\bs)=\times_{i=1} ^IQ_{\theta r}(\bs_i)\;(\bs=(\bs_1,\bs_2,\dots,\bs_I)$, $\Pi=\times_{i=1}^I \Pi_i$.
	Let us notice that $\Pi\subseteq\mathcal{D}_J$ with 
	\begin{equation*}
	J=\max_{1\le i\le I}\,j_i\le n.
	\end{equation*}
	It is easy to see that $Q_{\theta r}(\bs)\subseteq Q_r(\bz)\cap\Pi$. This means that the sequence $\{\mathcal{D}_N\}_{N=1}^\infty$ satisfies 
	condition (b) of Definition \ref{dfdenslogmtetpart1}.
\end{example}

Now we shall construct some counterexamples connected with conditions of discreteness of the spectrum
of the operator $H=-\Delta+V(\e)\cdot$, obtained above.

\begin{example}\label{ex4}
	Here we shall  construct an example of the potential $V(\e)\ge 0$ which satisfies conditions 
	of Theorem \ref{thlogmdenspolyhed} (hence the spectrum of the operator $H=-\Delta+V(\e)\cdot$ is discrete) , but the condition \eqref{cndGMD} of the criterion from \cite{4},
	formulated in Remark \ref{remGMD}, is not satisfied for it. 
	Let us return to the sequence $\{D_n\}_{n=1}^\infty$ of subsets of the interval $[0,\,1]$ and considered in Example \ref{ex1}, and the following 
	sequence of subsets of the cube $Q_1(0)$:
	\begin{equation}\label{dfcalD}
	\mathcal{D}_n=D_n\times[0,\,1]^{d-1}.
	\end{equation} 
	Consider also the translations of the cube $Q_1(0)$ and the sets
	$\mathcal{D}_n$ by the vectors $\vec l=(l_1,l_2,\dots,l_d)\in\Z^d$: $Q_1(\vec l)=Q_1(0)+\{\vec l\}$, 
	\begin{equation}\label{dfcalDvecl}
	\mathcal{D}_n(\vec l)=\mathcal{D}_n+\{\vec l\}.
	\end{equation}
	The arguments of Examples \ref{ex1} and \ref{ex2} imply that for any fixed $\vec l\in\Z^d$  the sequence $\{\mathcal{D}_n(\vec l)\}_{n=1}^\infty$ 
	forms in $Q_1(\vec l)$ a $(\log_3,\,1/9)$-dense system.  For $\beta\in(0,1)$ consider on $\R$ the
	$1$-periodic function $\theta_\beta(x)$, defined on the interval
	$(0,1]$ in the following manner:
	\begin{equation}\label{dfthatbetnew}
	\theta_\beta(x)=\left\{\begin{array}{ll}
	1&\quad\mathrm{for}\quad x\in(0,\beta],\\
	0&\quad\mathrm{for}\quad x\in(\beta,1].
	\end{array}\right.
	\end{equation}
	Let us take
	\begin{equation}\label{alphain}
	\alpha\in\Big(0,\,2\Big).
	\end{equation} 
	Consider the following function, defined on $(0,1]$:
	\begin{eqnarray}\label{dfSigmaaNalphnew}
	\hskip10mm\Sigma_{N,\,p,\alpha}(x):=\left\{\begin{array}{ll}
	0&\quad\mathrm{for}\quad x\in(0,\,1]\setminus\bigcup_{n=1}^\infty D_n,\\
N\theta_\beta(3^px)\vert_{\beta=3^{-\alpha n}}&\quad\mathrm{for}\quad x\in
	D_n\;(n=1,2,\dots)
	\end{array}\right.
	\end{eqnarray}
	$(N>0,\,p\in\N)$. Recall that we denote by  $P_1$ the operator, defined  by \eqref{dfP1}.
	Consider a function
	$\mathcal{N}:\Z^d\rightarrow\R_+$, satisfying the condition
\begin{equation}\label{cndNl1new}
\mathcal{N}(\vec l)\ge 1,\quad\mathcal{N}(\vec
l)\rightarrow\infty\quad\mathrm{for}\quad |\vec
l|_\infty\rightarrow\infty,
\end{equation}
	where $|\vec l|_\infty=\max_{1\le i\le d}|l_i|$.
	Let us construct the
	desired potential in the following manner:
	\begin{eqnarray}\label{dfpotentValnew}
	&&V_\alpha(\e):=\Sigma_{N,\,p,\,\alpha}(P_1(\e-\vec
	l))\vert_{N=\mathcal{N}(\vec l),\,p=|\vec l|_\infty+1}\quad
	\mathrm{for}\quad\vec l\in \Z^d\\
	&&\mathrm{and}\quad \e\in
	Q_1(\vec l).\nonumber
	\end{eqnarray}
	It is clear that $V_\alpha\in L_{\infty,\,loc}(\R^d)$.  
	Let us prove that the potential $V_\alpha(\e)$ satisfies all the conditions of Theorem \ref{thlogmdenspolyhed}. Let us take 
	a natural $n$, 
	\begin{equation}\label{jin1n}
	j\in\{1,2,\dots,n\}
	\end{equation}
	and a cube  
	\begin{equation}\label{inclcube}
	Q(\vec\xi,\,n)\subseteq\mathcal{D}_j(\vec l)
	\end{equation} 
	of the $3$-adic partition of $Q_1(\vec l)$. Let us notice that 
	\begin{equation*}
	P_1\big(Q(\vec\xi,\,n)\big)=[k\,3^{-n},\,(k+1)\,3^{-n}]
	\end{equation*}
	for some $k\in\Z$.
	Then taking into account definitions \eqref{dfcalD}-\eqref{dfSigmaaNalphnew}, \eqref{dfpotentValnew} and the $1$-periodicity of $\theta_\beta(t)$, we get for $|\vec l|_\infty>n$:
	\begin{eqnarray}\label{bigest}
	&&\mathrm{mes}_d\Big(\big\{\e\in Q(\vec\xi,n):\,V_\alpha(\e)>0\big\}\Big)=\\
	&&3^{-(d-1)n}\mathrm{mes}_1\Big(\big\{x\in[k\,3^{-n},\,(k+1)\,3^{-n}]:\,\nonumber\\
	&&\theta_\beta(3^p\,x)\vert_{\beta=3^{-\alpha j},\,p=|\vec l|_\infty+1}>0 \big\}\Big)=\nonumber\\
	&&3^{-(d-1)n}\mathrm{mes}_1\Big(\big\{x\in[0,\,3^{-n}]:\,\theta_\beta(3^p\,x)\vert_{\beta=3^{-\alpha j},\,p=|\vec l|_\infty+1}>0 \big\}\Big)=\nonumber\\
	&&\frac{3^{-(d-1)n}}{3^{|\vec l|_\infty+1}}\mathrm{mes}_1\Big(\big\{t\in[0,\,3^{|\vec l|_\infty+1-n}]:\,\theta_\beta(t)\vert_{\beta=3^{-\alpha j}}>0\big\}\Big)=\nonumber\\
	&& 3^{-dn}\mathrm{mes}_1\Big(\big\{t\in[0,\,1]:\,\theta_\beta(t)\vert_{\beta=3^{-\alpha j}}>0\big\}\Big)=\nonumber\\
	&&3^{-\alpha j}\mathrm{mes}_d\big(Q(\vec\xi,\,n)\big)\ge 3^{-\alpha n}\mathrm{mes}_d\big(Q(\vec\xi,\,n)\big).\nonumber
	\end{eqnarray}
	Therefore in view of definitions \eqref{dfSVyrdel}-\eqref{dfLVsry}, \eqref{dfSigmaaNalphnew} and \eqref{dfpotentValnew},
	\begin{equation*}
	\bar V_\alpha^\star\big(3^{-\alpha n},\,\vec\xi,\,n\big)\ge\mathcal{N}(\vec l),
	\end{equation*}
	if conditions \eqref{jin1n} and \eqref{inclcube} are satisfied.  This estimate and conditions \eqref{alphain}, \eqref{cndNl1new}-b imply that condition 
	\eqref{cndlogdenspolyhed} of Theorem \ref{thlogmdenspolyhed} is satisfied for the potential $V_\alpha(\e)$ with $\gamma(r)=r^\alpha$ satisfying condition \eqref{cndgammar}. Hence the spectrum of the operator 
	$H=-\Delta+V_\alpha(\e)\cdot$ is discrete.  
	
	Let us show that the condition \eqref{cndGMD} of the criterion from \cite{4} is not satisfied for the potential $V_\alpha(\e)$. To this end 
	for any natural $n$ consider $\vec l_n\in\Z^d$ such that $|\vec l_n|_\infty=n$ and take a cube $Q(\vec\xi_n,\,n)\subset\mathcal{D}_n(\vec l_n)$. In view 
	of \eqref{dfthatbetnew}, \eqref{dfSigmaaNalphnew}, \eqref{dfpotentValnew} and \eqref{cndNl1new}-a, $\int_{Q(\vec\xi_n,\,n)}V_\alpha(\e)\,\mathrm{d}\e>0$.
	Then taking $\delta>0$, we have:
	\begin{eqnarray*}
		&&\big\{\e\in Q(\vec\xi_n,\,n):\,V_\alpha(\e)\ge\frac{\delta}{\mathrm{mes}_d(Q(\vec\xi_n,\,n))}\int_{Q(\vec\xi_n,\,n)}V_\alpha(\e)\,\mathrm{d}\e\big\}\subseteq\\
		&&\big\{\e\in Q(\vec\xi_n,\,n):\,V_\alpha(\e)>0\big\}.
	\end{eqnarray*}
	On the other hand, if we take in \eqref{bigest} $j=n$ and $\vec\xi=\vec\xi_n$, we get
	\begin{equation*}
	\mathrm{mes}_d\Big(\big\{\e\in Q(\vec\xi_n,\,n):\,V_\alpha(\e)>0\big\}\Big)=3^{-\alpha n}\mathrm{mes}_d\big(Q(\vec\xi_n,\,n)\big).
	\end{equation*}
	These circumstances mean that the sequence 
	\begin{equation*}
	\frac{\mathrm{mes}_d\Big(\big\{\e\in Q(\vec\xi_n,\,n):\,V_\alpha(\e)\ge\frac{\delta}{\mathrm{mes}_d(Q(\vec\xi_n,\,n))}\int_{Q(\vec\xi_n,\,n)}V_\alpha(\e)\,\mathrm{d}\e\big\}\Big)}
	{\mathrm{mes}_d\big(Q(\vec\xi_n,\,n)\big)}
	\end{equation*}
	tends to zero as $n\rightarrow\infty$. This means that condition \eqref{cndGMD} is not satisfied for the potential $V_\alpha(\e)$. 
\end{example}

\begin{example}\label{excorpolyhed}
	Consider the potential $V_\alpha (\e)$ constructed in the previous Example \ref{ex4}, but now
	\begin{equation}\label{varyalph}
	2(d-2)/d<\alpha<2.
	\end{equation} 
	As we have shown there, the conditions of Theorem \ref{thlogmdenspolyhed} are satisfied for this potential (hence the
	spectrum of the operator $H=-\Delta+V_\alpha(\e)\cdot$ is discrete). My goal is to show that 
	condition \eqref{cndSVy} of Theorem \ref{lmmeasinstcap1} is not satisfied for $V_\alpha(\e)$.
	Denote $\psi(r)=r^\alpha$. Then in view of the left inequality of \eqref{varyalph}, $\lim_{r\downarrow 0}r^{-2(d-2)/d}\psi(r)<\infty$.
	Let us take a function $\hat\gamma(r)$ satisfying condition \eqref{cndtildgam} for some $r_0>0$. Then 
	$\limsup_{r\downarrow 0}\frac{\hat\gamma(r)}{\psi(r)}=\infty$.
Hence for some positive decreasing sequence $\{r_j\}_{j=1}^\infty$ tending to zero
	\begin{equation}\label{decrseqrj1}
	\lim_{j\rightarrow \infty}\frac{\hat\gamma(r_j)}{\psi(r_j)}=\infty.
	\end{equation}
	In order to prove that condition \eqref{cndSVy} is not satisfied for the potential $V_\alpha(\e)$, it is sufficient to find  a sequence 
	of points $\by_j\in\R^d$ such that
	\begin{equation}\label{findseq1}
	\lim_{j\rightarrow\infty}|\by_j|=\infty\quad\mathrm{and}\quad \limsup_{j\rightarrow\infty}\bar V_\alpha^\star\big(\hat\gamma(r_j)\mathrm{mes}_d(Q_{r_j}(\by_j)),\,Q_{r_j}(\by_j)\big)<\infty.
	\end{equation}
	Let us choose an increasing sequence of natural numbers $\{n_j\}_{j=1}^\infty$ such that 
	\begin{equation}\label{ineqforrj1}
	3^{-(n_j+1)}\le r_j<3^{-n_j}.
	\end{equation}
	Consider the vectors $\vec l\in\Z^d$ of the form $\vec l=(l,0,\dots,0)$, where $l\in\Z$ will be chosen below. Consider the intervals $I_{n_j,1}=[a_{n_j,1},\,b_{n_j,1}]\subset D_{n_j}$ and
	the cubes $Q_{j,l}=Q_{3^{-n_j}}(\tilde\by_{j,l})$, where $\tilde\by_{j,l}=\big(a_{n_j,1}+l,\,0,\dots,0\big)$. Then $P_1(Q_{j,l})=I_{n_j,1}+l$. In view of the right inequality
	of \eqref{ineqforrj1}, $Q_{r_j}(\tilde\by_{j,l})\subset Q_{j,l}$. Then using definitions \eqref{dfpotentValnew} and \eqref{dfthatbetnew}, \eqref{dfSigmaaNalphnew}, 
	Lemma \ref{lmpropsigmaNbetnew1}, the left inequality of \eqref{ineqforrj1} and taking $l=n_j$, we have:
	\begin{eqnarray}\label{estmespositVpsi1}
	&&\mathrm{mes}_d\big(\{\e\in Q_{r_j}(\tilde\by_{j,l}):\,V_\alpha(\e)>0\}\big)\le\\
	&&\mathrm{mes}_d\big(\{\e\in Q_{j,l}:\,V_\alpha(\e)>0\}\big)=\nonumber\\
	&&3^{-n_j(d-1)}\mathrm{mes}_1\big(\{x\in I_{n_j,1}:\,\theta_\beta(3^{l+1}x)\vert_{\beta=\psi(3^{-(n_j+1)})}\}\big)\le\nonumber\\
	&&\psi(3^{-(n_j+1)})(3^{-n_jd}+3\cdot 3^{-n_j(d-1)}3^{-l-1})=\nonumber\\
	&&2\psi(3^{-(n_j+1)})3^{-(n_j+1)d}3^d\le 2\cdot 3^d\psi(r_j)\mathrm{mes}_d(Q_{r_j}(\tilde\by_{j,l})).\nonumber
	\end{eqnarray} 
	On the other hand, in view of \eqref{decrseqrj1}, 
	\begin{equation*}
	\exists\, J>0\quad\forall\, j\ge J:\quad 2\cdot 3^d\psi(r_j)<\hat\gamma(r_j).
	\end{equation*}
	This circumstance, estimate \eqref{estmespositVpsi1} and definitions 
 \eqref{dfpotentValnew} and \eqref{dfthatbetnew}, \eqref{dfSigmaaNalphnew} imply that 
	\begin{equation*}
	\forall\,j\ge J:\quad \bar V_\alpha^\star\big(\hat\gamma(r_j)\mathrm{mes}_d(Q_{r_j}(\by_j)),\,Q_{r_j}(\by_j)\big)=0, 
	\end{equation*}
	where $\by_j=\tilde\by_{j,n_j}$. This means that relation \eqref{findseq1} is valid, i.e., the potential $V_\alpha(\e)$ does not satisfy condition \eqref{cndSVy}.
\end{example}

In the previous consideration we have used the following claim:

\begin{lemma}\label{lmpropsigmaNbetnew1}
	For any interval $[a,b]$ and $S>0$ consider the quantity
	\begin{equation*}
	M(S,\beta,a,b)=\mathrm{mes}_1\{x\in[a,b]:\,\theta_\beta(Sx)>0\},
	\end{equation*}
	where $\theta_\beta(x)$ is the $1$-periodic function, defined by
	\eqref{dfthatbetnew}. It satisfies the inequalities
	\begin{equation*}
		M(S,\beta,a,b)\le \beta(b-a)+2\beta/S,
	\end{equation*}
	\begin{equation*}
M(S,\beta,a,b)\ge \beta(b-a)-2\beta/S;
	\end{equation*}
\end{lemma}
\begin{proof}
	Taking into account the $1$-periodicity of $\theta_\beta(x)$, we obtain:
\begin{eqnarray*}
		&&M(S,\beta,a,b)\le M(S,\beta,[S a]/S,([S b]+1)/S)=\\
		&&S^{-1}M(1,\beta,[S a],[S b]+1)=\frac{\beta}{S}\big([S]+1-[S a]\big)\le\\
		&&\beta(b-a)+2\beta/S,
	\end{eqnarray*}
	\begin{eqnarray*}
		&&M(S,\beta,a,b)\ge M(S,\beta,([S a]+1)/S,[S b]/S)=\\
		&&S^{-1}M(1,\beta,[S a]+1,[S b])
		=\frac{\beta}{S}\big([S b]-1-[S a]\big)\ge\\
		&&\beta(b-a)-2\beta/S.
	\end{eqnarray*}
\end{proof}

\appendix
\section{Base polyhedron of harmonic capacity}
\label{sec:appendix}

\setcounter{equation}{0}

The following claim on the base polyhedron of harmonic capacity, defined by \eqref{defcore}, is valid:

\begin{proposition}\label{prcore}
	Suppose that a domain 
	\begin{equation*}
	\Omega\subset\R^d
	\end{equation*}
	 is open and bounded. Then the base polyhedron $\mathrm{BP}(\bar\Omega)$ 
	of the harmonic capacity on $\bar\Omega$  is nonempty, convex and weak*-compact.
	\end{proposition}
\begin{proof}
	As it is known (\cite{11})(p. 537), the harmonic capacity  on $\Omega$ is a
	non-negative, monotone and bounded set functions and
	$\mathrm{cap}(\emptyset)=0$. Recall that it is submodular, i.e., property \eqref{submod} is valid. Let us consider the set
	function which is called {\it dual} to ``cap'':
	$\mathrm{cap}^\star(A)=\mathrm{cap}(\bar\Omega)-\mathrm{cap}(A^c)\,(A\in\Sigma_B(\bar\Omega))$,
	where $A^c=\bar\Omega\setminus A$. It is easy to show that ``cap*''
	is a non-negative, monotone and bounded set function and
	$\mathrm{cap^\star}(\emptyset)=0$, but it is {\it supermodular}
	in the sense that for any pair of sets
	$A,\,B\in\Sigma_B(\bar\Omega)$ $\mathrm{cap}^\star(A\cup
	B)+\mathrm{cap}^\star(A\cap
	B)\ge\mathrm{cap}^\star(A)+\mathrm{cap}^\star(B)$. As it is known (\cite{10}, Proposition 1),  
for the set functions of this kind
	the collection of measures
	\begin{eqnarray}\label{dfCorest}
	&&\mathrm{Core}^\star(\bar\Omega))=\{\mu\in M(\bar\Omega)):\;\\
	&&\mu(A)\ge \mathrm{cap}^\star(A)\;\mathrm{for\; all}\;
	A\in\Sigma_B(\bar\Omega))\;\mathrm{and}\;\mu(\bar\Omega)=\mathrm{cap}^\star(\bar\Omega)\}\nonumber
	\end{eqnarray}
	is nonempty, convex and and weak*-compact. It is called the {\it
		core} of the set function ``cap*''.  Since ``cap*'' is non-negative, it
	is clear that $\mathrm{Core}^\star(\bar\Omega)\subseteq
	M^+(\bar\Omega)$. It is not difficult to show that
	$\mathrm{BP}(\bar\Omega)=\mathrm{Core}^\star(\bar\Omega))$
	(\cite{5}, Lemma 2.3). The claim is proven.
\end{proof}

Recall that we denote by $m_{d,r}$ the probability measure, generated on a ball $B_r(\by)$ by Lebesgue measure 
$\mathrm{mes}_d$ (see definition \eqref{dfmdr})  and by $\mathrm{BP}_{eq}(\by,\,r)$ we denote the set of all measures from
the base polyhedron $\mathrm{BP}(\bar B_r(\by))$ which are equivalent to  
$\mathrm{mes}_d$. Consider on $\bar B_r(\by)$  the normalized harmonic capacity
$\mathrm{cap}_n(A)=\frac{\mathrm{cap}(A)}{\mathrm{cap}(\bar
	B_r(\by))}\,(A\in\Sigma_B(\bar B_r(\by)))$.
The following claim describes a part of $\mathrm{BP}_{eq}(\by,\,r)$:

\begin{proposition}\label{lmdescribedens}
	The set $\mathrm{BP}_{eq}(\by,\,r)$ contains the set ${\mathrm
		M}_f(r,\by)$ of absolute continuous measures described in
	Definition \ref{dfMacry}.
\end{proposition}
\begin{proof}
	Since the isocapacity inequality 
	\begin{equation*}
	\mathrm{mes}_d(A)\le
	c_d\,(\mathrm{cap}(A))^{d/(d-2)}\quad (A\in\Sigma_B(\bar B_r(\by)))
	\end{equation*}
	 comes as identity for $A=\bar
	B_r(\by)$, it is equivalent to
\begin{equation}\label{equivtoisocap}
	f(m_{d,r}(A))\le\mathrm{cap}_n(A)\quad.
	\end{equation}
	Notice that the function $f(t)=t^{(d-2)/d}$ is increasing, concave and $f(1)=1$. Then the set function $f(m_{d,r})$ is a concave distortion of the probability
	measure $m_{d,r}$, hence it is submodular (\cite{2},
	\cite{10}). Therefore the conjugate to it $f^\star(m_{d,r})$ with
	$f^\star(t)=1-f(1-t)$ is a convex distortion of $m_{d,r}$, hence it is
	supermodular. Consider the core of $f^\star(m_{d,r})$, i.e., the
	following collection of  additive probability
	measures: $\mathrm{core}\big(f^\star(m_{d,r})\big)=\{P:\,P(A)\ge
	f^\star(m_{d,r}(A)),\;\forall\,A\in\Sigma_B(\bar B_r(\by))\}$, which coincides with
	the base polyhedron of $f(m_{d,r})\;$ $\mathrm{BP}(f(m_{d,r}))=\{P:\,P(A)\le
	f(m_{d,r}(A)),\;\forall\,A\in\Sigma_B(\bar B_r(\by))\}$ (\cite{5}). Hence in view
	of \eqref{equivtoisocap} and definition \eqref{defcore} with $\Omega=B_r(\by)$,
	\begin{equation}\label{inclinpolyhed}
	\mathrm{cap}(\bar
	B_r(0))\cdot\mathrm{core}\big(f^\star(m_{d,r})\big)\subseteq\mathrm{BP}(\by,r).
	\end{equation}
	 	On the other hand,
	since the Lebesque measure is non-atomic, we can use Theorem 2 of \cite{2}, which says
	that $\mathrm{core}\big(f^\star(m_{d,r})\big)$ consists of absolutely continuous measures on $\bar B_r(\by)$, whose densities run run over 
	the set 
	\begin{equation}\label{descrcorestar}
	\overline{\mathrm{co}}\Big(\{(f^\star)^\prime\,\circ\,s:\,s\in\mathcal{S}(\bar
	B_r(\by),\,m_{d,r})\}\Big).
	\end{equation}
	We have: $(f^\star)^\prime(t)=f^\prime(j(t))$, where $j(t)=1-t$.
	Since the mapping $j:\,[0,1]\rightarrow [0,1]$ is bijective and
	preserves Lebesgue measure, then the operator $J(s)=j\circ\,
	s$ maps bijectively the set $\mathcal{S}(\bar B_r(\by),\,m_{d,r})$ onto
	itself. Hence the set under the big brackets in \eqref{descrcorestar} coincides with
	$\{f^\prime\,\circ\,s:\,s\in\mathcal{S}(\bar B_r(\by),\,m_{d,r})\}$.
	This circumstance and inclusion \eqref{inclinpolyhed} imply that
	the set ${\mathcal Co}\,(\by,\,r)$, defined by \eqref{dfcalCo}, is
	contained in $\mathrm{BP}(\by,\,r)$. It is easy to show that
	${\mathcal Co}\,(\by,\,r)\subset L_1\big(B_r(\by),\,m_{d,r}\big)$.
	Furthermore, since $f^\prime(t)=((d-2)/d)t^{-2/d}$, then
	$\inf_{t\in[0,1]}f^\prime(t)>0$. Hence all the measures having the
	densities in ${\mathcal Co}\,(\by,\,r)$ are equivalent to the
	Lebesgue measure. The proposition is proven.
\end{proof}

\end{document}